\documentclass[11pt,leqno]{article}

\usepackage{amsmath,amsthm,amssymb}
\usepackage{authblk,subfigure,booktabs}
\usepackage{enumerate}
\usepackage{graphicx}
\usepackage{color}

\usepackage{mathabx,bm}
\usepackage{algorithm}
\usepackage[noend]{algpseudocode}
\usepackage{url}
\usepackage{enumitem}

\usepackage{geometry}
\usepackage{flafter,placeins}
\usepackage{listings}

\usepackage{apptools}

\theoremstyle{plain}
\newtheorem{thm}{Theorem}[section]
\newtheorem{prop}[thm]{Proposition}
\newtheorem{cor}[thm]{Corollary}
\newtheorem{lem}[thm]{Lemma}

\theoremstyle{remark}
\newtheorem{rem}[thm]{Remark}

\theoremstyle{definition}

\makeatletter

\numberwithin{equation}{section}
\numberwithin{figure}{section}
\numberwithin{table}{section}

\title{\Large\bf Convergence of kernel and neural-network methods \\ for Hamilton--Jacobi--Bellman equations on unbounded domains}

\author{Yumiharu Nakano\thanks{E-mail: nakano@comp.isct.ac.jp}}
\affil{Department of Mathematical and Computing Science \protect \\ Institute of Science Tokyo}

\date{October 7, 2026}

\begin{document}

\maketitle

\begin{abstract}
We derive error bounds from empirical squared residuals for two non-monotone,
mesh-free approximations of Hamilton--Jacobi--Bellman equations on unbounded
spatial domains: collocation with Wendland kernels under a native space norm
constraint, and physics-informed neural networks with smooth activations under
a Sobolev-norm constraint. The main step is to recover the pointwise residual
control needed for viscosity-solution error estimates from the empirical loss,
while quantifying the discretization or sampling error. For the kernel method,
we prove uniform convergence on compact sets under classical and native space
regularity and derive an a posteriori quantitative error bound for the unique
bounded continuous viscosity solution. For the neural-network method, we prove
uniform convergence on compact sets in probability and give a quantitative
high-probability error bound under classical Sobolev regularity. The bounds are
governed by the fill distance for kernels and by the network and sample sizes
for neural networks. The proofs combine empirical-to-population estimates, an
interpolation inequality that turns an $L^2$ residual into a sup norm bound,
and a controlled diffusion estimate for the error caused by truncating the domain.
The regularity assumptions used to obtain convergence from approximation
properties are additional to viscosity wellposedness. Numerical experiments
in one and two space dimensions illustrate the methods and the estimates.
\end{abstract}

\par


\section{Introduction}\label{sec:1}

We study non-monotone approximation schemes, based on minimizing an $L^2$-type residual, for
terminal value problems of Hamilton--Jacobi--Bellman (HJB) equations
\pagebreak[2]
\begin{equation}
\label{eq:1.1}
\left\{
\begin{aligned}
 &\partial_t v + \inf_{a\in A}H^a(t,x,D_xv(t,x),D^2_{xx}v(t,x))=0, \quad (t,x)\in
 [0,T)\times \mathbb{R}^d, \\
 &v(T,x)=g(x), \quad x\in\mathbb{R}^d,
\end{aligned}
\right.
\end{equation}
where $D_xv$ and $D^2_{xx}v$ denote the gradient and Hessian with respect to $x$, respectively, and
\[
 H^a(t,x,p,X)=\tfrac12\mathrm{tr}(\sigma\sigma^{\mathsf T}(t,x,a)X)+b(t,x,a)\cdot p+f(t,x,a),\qquad a\in A,
\]
is the Hamiltonian of a controlled diffusion. Here, $\mathrm{tr}$ denotes the matrix trace,
the superscript $\mathsf T$ denotes matrix transpose, and $A\subset\mathbb{R}^{d_1}$ is a nonempty compact
control set for some $d_1\in\mathbb{N}$. Its value function is generally only a viscosity
solution \cite{fle-son:2006,pha:2009}. Our residual-to-error estimate and the kernel a posteriori
bound apply directly to this viscosity solution. The convergence results based on approximating the
exact solution require, in addition, a classical solution with the stated native space or Sobolev regularity.

Classical methods for \eqref{eq:1.1} include finite differences \cite{kus-dup:2001,bon-zid:2003},
semi-Lagrangian schemes \cite{cam-fal:1995,deb-jak:2012}, probabilistic methods
\cite{pag-pha-pri:2004,fah-tou-war:2011,guo-etal:2015,nak:2014b,aki-fod:2018}, and kernel collocation
\cite{kan:1990b,nak:2017}. See \cite{fen-glo-nei:2013,nei-sal-zha:2017} for surveys of numerical
methods for fully nonlinear second-order equations. More recent developments include narrow-stencil
finite differences \cite{fen-lew:2021}, higher-order backward differentiation schemes under stronger
regularity \cite{bok-pic-rei:2021}, and boundary treatments and multigrid solvers for semi-Lagrangian
schemes \cite{rei-rot:2017}. Convergence to
the viscosity solution is usually derived from the framework of Barles and Souganidis
\cite{bar-sou:1991}, which requires monotonicity. For non-monotone schemes, rigorous theories exist
in restricted settings: discontinuous Galerkin methods with optimal $H^2$ rates for HJB equations
under the Cord{\`e}s condition \cite{sme-sul:2013,sme-sul:2014,sme-sul:2016}, and
$L^{\infty}$-convergent minimal-residual methods based on the Alexandrov--Bakelman--Pucci principle
\cite{gal-tra:2024}. These theories are
confined to bounded domains, to the Cord{\`e}s or ABP operator classes, and to mesh-based
discretizations.

Among mesh-free methods, deep BSDE schemes and related approaches use stochastic representations
\cite{e-etal:2017,han-jen-e:2018,hur-pha-war:2020,bac-hur-pha-lan:2022,hut-jen-kru-ngu:2020},
whereas physics-informed neural networks (PINNs) and the deep Galerkin method (DGM) minimize an
empirical PDE residual \cite{sir-spi:2018,rai-per-kar:2019}. Convergence is known for quasilinear
equations under uniform $C^{1,2}$-bounds, and error estimates have been developed for linear and
semilinear problems \cite{sir-spi:2018,mis-mol:2022,dr-mis:2022,dr-mis:2024}. Recent theories based
on coercivity, optimal recovery or $\Gamma$-convergence likewise concern linear elliptic and
parabolic equations \cite{zei-mas-mar:2025,bon-dev-pet-sie:2026,gaz-gka-mak:2025}.

For HJB equations, available analyses treat time-stepping meshfree collocation
\cite{nak:2017}, DGM for finite-state mean-field control \cite{hof-sun-coh:2024}, RKHS recovery for
first-order infinite-horizon problems \cite{ehr-azm-haa:2026}, or neural policy iteration through a
sequence of linear PDEs \cite{men-etal:2024,kim-etal:2026pi}. Residual-to-error certificates are also
available for stationary first-order HJB equations \cite{liu:2026}. These settings differ from the
direct minimization of an empirical residual for the finite-horizon second-order HJB equation on
$\mathbb{R}^d$ considered here. Building on the pointwise-consistency and controlled-diffusion
truncation framework of \cite{nak:2025x}, the main advance of this paper is to recover the required
pointwise residual control from computable empirical $L^2$ losses under norm constraints, while
quantifying the discretization error for kernels and the sampling error for neural networks. This
yields a unified conditional analysis in which the resulting solution-error bounds are expressed in
terms of the fill distance for kernels and the network and sample sizes for neural networks.

The difficulty in carrying out this program comes from a mismatch of error notions. Kernel
collocation and PINN/DGM minimize an empirical squared PDE residual and terminal mismatch (defined
precisely in Sections~\ref{sec:3}--\ref{sec:4}), whereas viscosity-solution error estimates require
pointwise control of these quantities. We bridge this gap by constraining the candidate in a norm that supplies
locally uniform $C^3$-control through Sobolev embedding, and by passing from the $L^2$-loss to the
sup-norm of the residual through a Gagliardo--Nirenberg-type interpolation between $L^2$ and
H\"older spaces; finally, the stochastic control representation of $v$ yields a controlled-diffusion tail estimate
that converts a residual bound on a truncation cylinder into convergence on compact subsets of
$[0,T]\times\mathbb{R}^d$.

Our contributions are as follows.
\begin{enumerate}
\item[(i)] \emph{Empirical residual-to-error mechanism} (Sections~\ref{sec:2}--\ref{sec:4}).
Norm constraints give uniform regularity of the candidates, after which interpolation converts
population $L^2$ residuals into the pointwise residuals required by the controlled-diffusion
truncation estimate. The population residuals are bounded by empirical losses through a
fill-distance quadrature estimate for kernels and a uniform generalization estimate for neural
networks.
\item[(ii)] \emph{Kernel realization} (Section~\ref{sec:3}). The ansatz is a combination
$\sum_j\theta_j\Phi(\cdot-y^j)$ of translates of a Wendland function over a uniform space-time grid,
under the native space constraint $\theta^{\mathsf T}K\theta\le\lambda^2$. Under the additional
classical and native space regularity assumptions, Theorem~\ref{thm:ker.qual} gives uniform
convergence on compact sets. Under \textnormal{(A1)} and
\textnormal{(A2)}, Theorem~\ref{thm:ker.HJB} gives an a posteriori quantitative error bound for the
unique bounded continuous viscosity solution, governed by the empirical loss and fill distance.
\item[(iii)] \emph{Neural-network realization} (Section~\ref{sec:4}). The ansatz is a feedforward
network with a bounded smooth activation, such as $\tanh$, constrained in an $H^s$-norm with
$s>3+(d+1)/2$. Theorems~\ref{thm:3.1} and~\ref{thm:3.2} are the qualitative and quantitative
counterparts, the rate being governed by the approximation power of the network class (explicit for
$\tanh$ by \cite{drm:2021}) and by a Rademacher bound obtained from a covering argument in the
network parameters.
\end{enumerate}

The results above require different levels of regularity. The basic residual-to-error estimate and
the kernel a posteriori bound require only the unique
bounded continuous viscosity solution supplied by \textnormal{(A1)} and \textnormal{(A2)}; in particular,
they do not require uniform nondegeneracy. To derive convergence from approximation properties of
the kernel and network classes, we assume that the value function is classical and has the stated
native space or Sobolev regularity and decay. Classical regularity follows from Krylov's theory under
uniform nondegeneracy and sufficiently smooth data \cite{kry:1987,fle-son:2006}, but the additional
function-space regularity is imposed separately. Fully nonlinear equations without a stochastic
control representation are not treated.

Section~\ref{sec:2} fixes the setting and proves the two estimates shared by both realizations,
Sections~\ref{sec:3} and~\ref{sec:4} treat the two realizations, and Section~\ref{sec:5} reports
numerical experiments in $d=1,2$ that illustrate the methods.

\smallskip
\noindent\emph{Notation.}
We use the Euclidean (respectively Frobenius) norm $|\cdot|$ for vectors
(respectively matrices), and set $\mathbb N_0=\{0,1,2,\ldots\}$. For
$m\in\mathbb N$, a multi-index $\alpha=(\alpha_1,\ldots,\alpha_m)\in\mathbb N_0^m$,
and $y=(y_1,\ldots,y_m)\in\mathbb R^m$, we write
\[
 D^\alpha z
 =\partial_{y_1}^{\alpha_1}\cdots\partial_{y_m}^{\alpha_m}z,
 \qquad
 |\alpha|_1=\alpha_1+\cdots+\alpha_m.
\]
For an open set $\mathcal U\subset\mathbb R^m$ and $k\in\mathbb N_0$,
$C^k(\overline{\mathcal U})$ denotes the functions whose derivatives up to
order $k$ extend continuously to $\overline{\mathcal U}$, and
\[
 \|z\|_{C^{k}(\overline{\mathcal U})}
 :=\sum_{|\alpha|_1\le k}
   \sup_{y\in\overline{\mathcal U}}|D^\alpha z(y)|.
\]
We write $C_b^k(\overline{\mathcal U})$ when this norm is finite and use the
standard Sobolev spaces $H^s(\mathcal U)$. For $\gamma\in(0,1]$, we set
\[
 [z]_{C^{0,\gamma}(\overline{\mathcal U})}
 :=\sup_{\substack{y,y'\in\overline{\mathcal U}\\y\ne y'}}
 \frac{|z(y)-z(y')|}{|y-y'|^\gamma}.
\]
For functions of $(t,x)$, $C^{\nu,\kappa}$ denotes continuous differentiability
up to order $\nu$ in $t$ and up to order $\kappa$ in $x$.

\section{Setting and two basic estimates}\label{sec:2}

\subsection{Problem setting and truncation domains}\label{sec:2.1}

Throughout the paper, for $t\in[0,T]$, $x,p\in\mathbb{R}^d$, $X\in\mathbb{S}^d$ (the symmetric $d\times d$
matrices) and $a\in A$,
\begin{equation}
\label{eq:2.HJB}
 H^a(t,x,p,X)=\tfrac{1}{2}\mathrm{tr}(\sigma\sigma^{\mathsf{T}}(t,x,a)X)+b(t,x,a)\cdot p+f(t,x,a),
 \qquad H:=\inf_{a\in A}H^a,
\end{equation}
so that \eqref{eq:1.1} reads $\partial_tv+H(t,x,D_xv,D^2_{xx}v)=0$. We assume the following.
\begin{enumerate}
\item[\textnormal{(A1)}] The control set $A\subset\mathbb{R}^{d_1}$ is nonempty and compact, and
$b:[0,T]\times\mathbb{R}^d\times A\to\mathbb{R}^d$,
$\sigma:[0,T]\times\mathbb{R}^d\times A\to\mathbb{R}^{d\times d_2}$ and $f:[0,T]\times\mathbb{R}^d\times A\to\mathbb{R}$
are continuous. There is a constant $C_1>0$ such that for $\psi=b,\sigma,f$, $t,t'\in[0,T]$,
$x,x'\in\mathbb{R}^{d}$ and $a\in A$,
\(
 |\psi(t,x,a)|\le C_1
\) and
\(
 |\psi(t,x,a)-\psi(t',x',a)|\le C_1(|t-t'|+|x-x'|).
\)
\item[\textnormal{(A2)}] The function $g$ is Lipschitz continuous and bounded on $\mathbb{R}^d$.
\end{enumerate}

Fix $\bar x\in\mathbb{R}^d$ and a sequence of radii $(R_n)_{n\ge1}$ such that
$R_n>0$ and $R_n\to\infty$. The truncation cylinders, the
residual of a function $\varphi\in C^{1,2}$ are
\begin{equation}
\label{eq:Qn}
 Q_n=(0,T)\times B_{R_n},\qquad
 B_{R_n}:=\{x\in\mathbb{R}^d: \max_{1\le i\le d}|x_i-\bar{x}_i|<R_n\},
\end{equation}
\begin{equation}
\label{eq:residual}
 r(\varphi)(t,x)=\partial_t\varphi(t,x)+H(t,x,D_x\varphi(t,x),D^2_{xx}\varphi(t,x)),
\end{equation}
Let $\mathrm{Leb}_m$ denote the $m$-dimensional Lebesgue measure. The sampling measures used by
both realizations are
\[
 \mu_n:=\frac{\left.\mathrm{Leb}_{d+1}\right|_{Q_n}}
                  {\mathrm{Leb}_{d+1}(Q_n)},
 \qquad
 \mu_n^{\rm T}:=\frac{\left.\mathrm{Leb}_{d}\right|_{B_{R_n}}}
                         {\mathrm{Leb}_{d}(B_{R_n})}.
\]
For $R_n\ge1$, the Sobolev embeddings used below have constants independent of $n$: one can
extend functions from these boxes by reflection in each coordinate, with a fixed collar width
depending only on $T$, $d$ and the Sobolev order, and then apply the embedding on
$\mathbb{R}^{d+1}$ (or $\mathbb{R}^d$ for the terminal slice).

\begin{lem}[Hamiltonian and residual]\label{lem:2.F}
Assume \textnormal{(A1)}. There is a constant $C$, depending only on $C_1$ and $d$, such that the
following hold for all $t,t'\in[0,T]$, $x,x'\in\mathbb{R}^d$, $p,p'\in\mathbb{R}^d$ and $X,X'\in\mathbb{S}^d$.
\begin{enumerate}
\item[\textnormal{(i)}] $|H(t,x,p,X)|\le C(1+|p|+|X|)$.
\item[\textnormal{(ii)}] $|H(t,x,p,X)-H(t',x',p',X')|\le C(1+|p|+|X|)(|t-t'|+|x-x'|)+C(|p-p'|+|X-X'|)$.
\item[\textnormal{(iii)}] If $\mathcal{O}\subset\mathbb{R}^{d+1}$ is bounded, open and convex, $M\ge0$, and
$\varphi\in C^3(\overline{\mathcal{O}})$ satisfies $\|\varphi\|_{C^3(\overline{\mathcal{O}})}\le M$, then
$[r(\varphi)]_{C^{0,1}(\overline{\mathcal{O}})}\le C(1+M)$.
\end{enumerate}
\end{lem}
\begin{proof}
For a fixed $a\in A$, the bounds in \textnormal{(A1)} give
\[
 |\sigma\sigma^{\mathsf T}(t,x,a)|\le C_1^2
\]
and
\[
 |\sigma\sigma^{\mathsf T}(t,x,a)-\sigma\sigma^{\mathsf T}(t',x',a)|
 \le 2C_1^2(|t-t'|+|x-x'|).
\]
Together with $|\mathrm{tr}(SX)|\le|S||X|$, these estimates give (i) and (ii) for $H^a$, with a
constant independent of $a$. The same bounds hold for $H=\inf_aH^a$ because
$|\inf_a H^a-\inf_a G^a|\le\sup_a|H^a-G^a|$.

It remains to prove (iii). Let $y,y'\in\overline{\mathcal O}$. Since $\mathcal O$ is convex, the
line segment from $y$ to $y'$ lies in $\overline{\mathcal O}$. The mean value theorem and the bound
$\|\varphi\|_{C^3(\overline{\mathcal O})}\le M$ give
\[
 |\partial_t\varphi(y)-\partial_t\varphi(y')|
 +|D_x\varphi(y)-D_x\varphi(y')|
 +|D^2_{xx}\varphi(y)-D^2_{xx}\varphi(y')|
 \le CM|y-y'|.
\]
Also, $|D_x\varphi|+|D^2_{xx}\varphi|\le CM$ on $\overline{\mathcal O}$. Applying (ii) in the
definition of $r(\varphi)$ now gives
$|r(\varphi)(y)-r(\varphi)(y')|\le C(1+M)|y-y'|$.
\end{proof}

Under \textnormal{(A1)} and \textnormal{(A2)}, the equation \eqref{eq:1.1} has a unique bounded continuous
viscosity solution; see Pham~\cite[Chapter~4]{pha:2009} and Fleming and
Soner~\cite{fle-son:2006}. Hence every bounded classical solution coincides with this solution. We
will also use the corresponding uniqueness result when a bounded Lipschitz function of $(t,x)$ is
added to $f$.

Uniform nondegeneracy and sufficiently smooth coefficients yield a classical solution by Krylov's
regularity theory \cite{kry:1987,fle-son:2006}. The Sobolev regularity imposed in
Sections~\ref{sec:3} and~\ref{sec:4} is a separate assumption, because it also requires smoothness and
decay at infinity. The test problem in Section~\ref{sec:5.1} satisfies this assumption.

\subsection{From $L^2$ to sup-norm bounds}\label{sec:2.2}

We first record an elementary interpolation estimate that converts an $L^2$ bound into a
sup-norm bound under H\"older control.

\begin{lem}[$L^2$--H\"older interpolation on boxes]\label{lem:2.1}
Let $\mathcal{O}=\prod_{i=1}^m(a_i,b_i)\subset\mathbb{R}^m$ be a bounded open box, put
$r_0:=\min\{1,\min_i(b_i-a_i)\}$, and let $\alpha\in(0,1]$. There is a constant $C_0=C_0(m,\alpha)$ such that,
for every $f\in C^{0,\alpha}(\overline{\mathcal{O}})$,
\begin{equation}
\label{eq:2.1.GN}
 \|f\|_{L^{\infty}(\mathcal{O})}
   \le C_0\,r_0^{-m/2}\Big(\,\|f\|_{L^2(\mathcal{O})}^{2\alpha/(2\alpha+m)}\,[f]_{C^{0,\alpha}(\overline{\mathcal{O}})}^{m/(2\alpha+m)}
            \;+\;\|f\|_{L^2(\mathcal{O})}\,\Big).
\end{equation}
In particular, for $R_n\ge1/2$ the constant in \eqref{eq:2.1.GN} is bounded independently of $n$ for
the cylinders $\mathcal{O}=Q_n\subset\mathbb{R}^{d+1}$ and for the cubes $\mathcal{O}=B_{R_n}\subset\mathbb{R}^d$,
for which $r_0=\min\{1,T\}$ and $r_0=1$, respectively.
\end{lem}
\begin{proof}
Step (i). Let $c_0:=(2\sqrt m)^{-m}$. We claim that
\begin{equation}
\label{eq:2.1.cone}
 |B(x_0,r)\cap\mathcal{O}|\ge c_0\,r^m,\qquad x_0\in\overline{\mathcal{O}},\;\;r\in(0,r_0].
\end{equation}
For each $i$, $(x_{0,i}-a_i)+(b_i-x_{0,i})=b_i-a_i\ge r$, so $b_i-x_{0,i}\ge r/2$ or $x_{0,i}-a_i\ge r/2$;
let $\omega_i:=+1$ in the first case and $\omega_i:=-1$ otherwise. With $s:=r/(2\sqrt m)\le r/2$, the
cube $\{x:\omega_i(x_i-x_{0,i})\in(0,s),\ i=1,\ldots,m\}$ is contained in $\mathcal{O}$ and, as its points
satisfy $|x-x_0|\le s\sqrt m<r$, in $B(x_0,r)$. The volume of this cube is $s^m=c_0r^m$.

Step (ii). 
Set $M:=\|f\|_{L^{\infty}(\mathcal{O})}$, $L:=[f]_{C^{0,\alpha}(\overline{\mathcal{O}})}$ and
$E:=\|f\|_{L^2(\mathcal{O})}$, and assume $M>0$. Choose $x_0\in\overline{\mathcal{O}}$ with $|f(x_0)|=M$.
For $r\in(0,r_0]$ with $Lr^{\alpha}\le M/2$ we have $|f|\ge M/2$ on $B(x_0,r)\cap\mathcal{O}$, hence, by
\eqref{eq:2.1.cone},
\begin{equation}
\label{eq:2.1.lb}
 E^2\ge\tfrac14 M^2c_0\,r^m .
\end{equation}
If $r_*:=(M/(2L))^{1/\alpha}\le r_0$ (with $L>0$), taking $r=r_*$ in \eqref{eq:2.1.lb} and solving for $M$ gives
$M\le(4\cdot2^{m/\alpha}/c_0)^{\alpha/(2\alpha+m)}L^{m/(2\alpha+m)}E^{2\alpha/(2\alpha+m)}$. Otherwise
$Lr_0^{\alpha}\le M/2$, and $r=r_0$ in \eqref{eq:2.1.lb} gives $M\le 2c_0^{-1/2}r_0^{-m/2}E$. Since
$r_0\le1$, \eqref{eq:2.1.GN} follows in both cases.
\end{proof}

\begin{rem}\label{rem:2.1}
We use \eqref{eq:2.1.GN} with $\alpha=1$, with $m=d+1$ for the residual on $Q_n$ and $m=d$ for the
terminal mismatch on $B_{R_n}$; this produces the exponents $1/(d+3)$ and $1/(d+2)$ in
\eqref{eq:3.kernel.HJB} and \eqref{eq:3.5.eta}. A higher-order H\"older or Sobolev norm of the residual
is not available in general, because $H$ is only Lipschitz in $(p,X)$; the Lipschitz bound of
Lemma~\ref{lem:2.F}\,(iii) is the natural input.
\end{rem}

\subsection{From residual bounds to approximation error}\label{sec:2.3}

The rate analyses of Sections~\ref{sec:3.5} and~\ref{sec:4.5} are carried out on the bounded
cylinders \eqref{eq:Qn}, whereas \eqref{eq:1.1} is posed on all of $\mathbb{R}^d$. The following
estimate converts an interior residual bound into a global sup-norm error, at the price of an
exponentially small tail; it is the only place where the stochastic representation of the value
function is used. It is adapted from the corresponding result of \cite{nak:2025x} and proved here so that the present
paper is self-contained.

Throughout this subsection \textnormal{(A1)} and \textnormal{(A2)} are in force.
Let $\lambda_n\ge1$, and let $v_n\in C^3_b([0,T]\times\mathbb{R}^d)$ satisfy
$\|v_n\|_{C^{1,2}([0,T]\times\mathbb{R}^d)}\le\lambda_n$. Put
\begin{equation}
\label{eq:tail.eta}
 \eta_n^{\rm res}:=\sup_{(t,x)\in[0,T]\times B_{R_n}}|r(v_n)(t,x)|,
 \qquad
 \eta_n^{\rm T}:=\sup_{x\in B_{R_n}}|v_n(T,x)-g(x)|,
 \qquad
 \eta_n:=\eta_n^{\rm res}\vee\eta_n^{\rm T}.
\end{equation}

The following theorem bounds the approximation error in terms of the residual and terminal errors.
\begin{thm}\label{thm:tail}
Suppose that \textnormal{(A1)} and \textnormal{(A2)} hold. For each $n\ge1$, let
$v_n\in C^3_b([0,T]\times\mathbb{R}^d)$ satisfy
\[
 \|v_n\|_{C^{1,2}([0,T]\times\mathbb{R}^d)}\le\lambda_n.
\]
Then for every compact
$Q\subset[0,T]\times\mathbb{R}^d$ there are $C>0$ and $n_0\in\mathbb{N}$ such that
\begin{equation}
\label{eq:tail.bound}
\begin{aligned}
 \sup_{(t,x)\in Q}|v(t,x)-v_n(t,x)|
 &\le T\eta_n^{\rm res}+\eta_n^{\rm T}
   +C\lambda_n\exp\Bigl(-\frac{R_n^2}{8dC_1^2T}\Bigr) \\
 &\le(1+T)\eta_n
   +C\lambda_n\exp\Bigl(-\frac{R_n^2}{8dC_1^2T}\Bigr),\qquad n\ge n_0,
\end{aligned}
\end{equation}
where $v$ is the unique bounded continuous viscosity solution of \eqref{eq:1.1}.
\end{thm}

\begin{proof}
The value function in stochastic control is the unique bounded continuous viscosity solution of
\eqref{eq:1.1} under \textnormal{(A1)} and \textnormal{(A2)}; see, for example,
Pham~\cite[Chapter~4]{pha:2009}. It therefore equals $v$, and
\begin{equation}
\label{eq:tail.rep}
 v(t,x)=\inf_{\alpha\in\mathcal{A}}\mathbb{E}\Bigl[\int_t^Tf(s,X_s^{t,x,\alpha},\alpha_s)\,ds
   +g(X_T^{t,x,\alpha})\Bigr],
\end{equation}
Here $X^{t,x,\alpha}$ solves
$dX_s=b(s,X_s,\alpha_s)\,ds+\sigma(s,X_s,\alpha_s)\,dW_s$, $X_t=x$, on a filtered probability space
carrying a $d_2$-dimensional Brownian motion $W$, and $\mathcal{A}$ is the set of $A$-valued
progressively measurable controls.

Since $v_n\in C^3_b([0,T]\times\mathbb{R}^d)$, the derivatives $\partial_tv_n$, $D_xv_n$, $D^2_{xx}v_n$ are bounded and
Lipschitz, so Lemma~\ref{lem:2.F}\,(i)--(ii) shows that $r_n:=r(v_n)$ is bounded and Lipschitz on $[0,T]\times\mathbb{R}^d$, and
\[
 \partial_tv_n+\inf_{a\in A}\bigl\{H^a(t,x,D_xv_n,D^2_{xx}v_n)-r_n\bigr\}=0\quad\text{on }[0,T)\times\mathbb{R}^d.
\]
Since $v_n$ is a bounded
classical solution of this equation, the uniqueness statement of Section~\ref{sec:2.1}, applied with $f$ replaced by $f-r_n$ and $g$ by
$v_n(T,\cdot)$, which satisfy \textnormal{(A1)} and \textnormal{(A2)}, gives
\[
 v_n(t,x)=\inf_{\alpha\in\mathcal{A}}\mathbb{E}\Bigl[\int_t^T\bigl(f(s,X_s^{t,x,\alpha},\alpha_s)
   -r_n(s,X_s^{t,x,\alpha})\bigr)ds+v_n(T,X_T^{t,x,\alpha})\Bigr].
\]
Subtracting \eqref{eq:tail.rep} and using $|\inf_\alpha\Phi-\inf_\alpha\Psi|\le\sup_\alpha|\Phi-\Psi|$,
\[
 |v(t,x)-v_n(t,x)|\le\sup_{\alpha}\int_t^T\mathbb{E}\bigl|r_n(s,X_s^{t,x,\alpha})\bigr|ds
   +\sup_{\alpha}\mathbb{E}\bigl|g(X_T^{t,x,\alpha})-v_n(T,X_T^{t,x,\alpha})\bigr|.
\]
Since $B_{R_n}$ is defined with the $\ell^\infty$-norm, leaving it implies
$|X_s^{t,x,\alpha}-\bar x|\ge R_n$ in the Euclidean norm. Splitting each expectation according to whether the state lies in $B_{R_n}$ bounds the two
contributions from $B_{R_n}$ by $T\eta_n^{\rm res}$ and $\eta_n^{\rm T}$ respectively, by \eqref{eq:tail.eta}, while on the complement both integrands are bounded by $C(1+\|v_n\|_{C^{1,2}})\le C\lambda_n$, by
Lemma~\ref{lem:2.F}\,(i), \textnormal{(A2)} and $\lambda_n\ge1$. Hence
\begin{equation}
\label{eq:tail.split}
 |v(t,x)-v_n(t,x)|\le T\eta_n^{\rm res}+\eta_n^{\rm T}
  +C\lambda_n\sup_{\alpha\in\mathcal{A}}
   \mathbb{P}\Bigl(\sup_{t\le s\le T}|X_s^{t,x,\alpha}-\bar x|\ge R_n\Bigr).
\end{equation}
It remains to show that, for some $n_0$,
\begin{equation}
\label{eq:tail.tail}
 \sup_{(t,x)\in Q}\ \sup_{\alpha\in\mathcal{A}}
 \mathbb{P}\Bigl(\sup_{t\le s\le T}|X_s^{t,x,\alpha}-\bar x|\ge R_n\Bigr)
  \le 2d\exp\Bigl(-\frac{R_n^2}{8dC_1^2T}\Bigr),\qquad n\ge n_0 .
\end{equation}
Fix $t,x,\alpha$ and set $b_s=b(s,X_s^{t,x,\alpha},\alpha_s)$,
$\sigma_s=\sigma(s,X_s^{t,x,\alpha},\alpha_s)$. Then
\[
 \Bigl\{\sup_{t\le s\le T}|X_s^{t,x,\alpha}-\bar x|\ge R_n\Bigr\}
 \subset\Bigl\{|\bar x|+|x|+\int_t^T|b_s|\,ds\ge\tfrac{R_n}{2}\Bigr\}
 \cup\Bigl\{\sup_{t\le s\le T}\Bigl|\int_t^s\sigma_u\,dW_u\Bigr|\ge\tfrac{R_n}{2}\Bigr\}.
\]
Because $R_n\to\infty$, there is $n_0$ with $R_n>2(|\bar x|+\sup_{(s,y)\in Q}|y|+C_1T)$ for
$n\ge n_0$, so the first event is empty by \textnormal{(A1)}, and only the $\mathbb{R}^d$-valued
martingale $M_s:=\int_t^s\sigma_u\,dW_u$ contributes. If $|M_s|\ge R_n/2$, then
$|M^i_s|\ge\rho:=R_n/(2\sqrt d)$ for some coordinate $i$. Each $M^i$ is a continuous martingale with
$\langle M^i\rangle_T\le C_1^2T$, because $|\sigma|\le C_1$. Let $\tau:=\inf\{s\in[t,T]:M^i_s\ge\rho\}\wedge T$.
Optional sampling applied to the exponential martingale $\exp(\mu M^i-\tfrac{\mu^2}{2}\langle M^i\rangle)$,
$\mu>0$, gives
\[
 1=\mathbb{E}\Bigl[\exp\Bigl(\mu M^i_{\tau}-\tfrac{\mu^2}{2}\langle M^i\rangle_{\tau}\Bigr)\Bigr]
  \ge\mathbb{P}\Bigl(\sup_{t\le s\le T}M^i_s\ge\rho\Bigr)\exp\Bigl(\mu\rho-\tfrac{\mu^2}{2}C_1^2T\Bigr),
\]
and the choice $\mu=\rho/(C_1^2T)$ yields
$\mathbb{P}(\sup_{s}M^i_s\ge\rho)\le\exp(-\rho^2/(2C_1^2T))=\exp(-R_n^2/(8dC_1^2T))$. Applying the same
bound to $-M^i$ and summing over $i=1,\ldots,d$ gives \eqref{eq:tail.tail}. Combining with
\eqref{eq:tail.split} gives \eqref{eq:tail.bound}.
\end{proof}

\section{Kernel realization}\label{sec:3}

We now apply the tools of Section~\ref{sec:2} to kernel-based collocation
on a positive definite radial basis function $\Phi$ of Wendland type.
The candidate sequence $\{v_n\}$ is parameterized by a coefficient vector $\theta^{(n)}\in\mathbb{R}^N$
as a linear combination of kernel translates over a uniform space-time grid.
The native space norm constraint supplies locally uniform $C^3$-control through Sobolev embedding,
and the empirical $L^2$-loss on the grid is converted into a bound on the residual by combining the
Wendland kernel interpolation error estimate of \cite[Corollary~11.33]{wen:2005} with a quadrature
bound.

\subsection{Kernel ansatz and native space}\label{sec:3.1}

Let $\Phi:\mathbb{R}^{d+1}\to\mathbb{R}$ be a compactly supported, positive definite radial
basis function of the form $\Phi(y)=\phi(|y|)$ whose Fourier transform satisfies
\begin{equation}
\label{eq:3.decay}
 c_1(1+|\xi|^2)^{-\tau}\le\widehat{\Phi}(\xi)\le c_2(1+|\xi|^2)^{-\tau},
 \qquad \xi\in\mathbb{R}^{d+1},
\end{equation}
for some constants $0<c_1\le c_2$ and a smoothness parameter
\begin{equation}
\label{eq:3.tau}
 \tau>3+\frac{d+1}{2}.
\end{equation}
By \cite[Theorem~10.12]{wen:2005}, the native (reproducing-kernel Hilbert) space $\mathcal{H}$
of $\Phi$ then coincides with the Sobolev space $H^{\tau}(\mathbb{R}^{d+1})$ up to equivalence
of norms; in particular, since $\widehat{\Phi}\ge 0$, one has
$\Phi\in C^{2j}(\mathbb{R}^{d+1})$ for every integer $j$ with $2j<2\tau-(d+1)$.
Since $\tau>3+(d+1)/2$, the Sobolev embedding $H^{\tau}(\mathbb{R}^{d+1})\hookrightarrow C^3_b(\mathbb{R}^{d+1})$
gives, for every $f\in\mathcal{H}$,
\begin{equation}
\label{eq:3.embed}
 \|f\|_{C^{3}_b(\mathbb{R}^{d+1})}\le C\,\|f\|_{\mathcal{H}},
\end{equation}
with a constant $C$ that depends only on the kernel.

With $Q_n$ and $B_{R_n}$ as in \eqref{eq:Qn}, the (anisotropic but uniform) collocation grid
$\Gamma_n\subset\overline{Q_n}$ is the tensor-product grid whose time coordinates are
$kT/N^t_n$, $k=0,\ldots,N^t_n$, and whose spatial coordinates along the $i$th axis are
$\bar x_i-R_n+2R_nj/(N^x_n-1)$, $j=0,\ldots,N^x_n-1$. Thus the grid contains both endpoints in
each coordinate and has $N=N_n=(N^t_n+1)(N^x_n)^d$ total points
$\{y^j\}_{j=1}^{N}=\{(\tilde t^j,\tilde x^j)\}_{j=1}^N$.
Let $Q:=(0,T)\times(-1,1)^d$ and let $\iota_n(s,z):=(s,\bar x+R_nz)$ be the affine map of $\overline Q$
onto $\overline{Q_n}$. We measure the resolution of $\Gamma_n$ by the fill distance of the rescaled grid,
$h_n:=\sup_{y\in Q}\min_{j}|y-\iota_n^{-1}(y^j)|$; since $\iota_n$ is Lipschitz with constant
$\max\{1,R_n\}$, the fill distance of $\Gamma_n$ in $Q_n$ is at most $R_nh_n$ when $R_n\ge1$.
The kernel matrix is
$K=(\Phi(y^i-y^j))_{1\le i,j\le N}\in\mathbb{R}^{N\times N}$, and the candidate approximation to
the value function $v$ is
\begin{equation}
\label{eq:3.kernel.ansatz}
 v_n^{(\lambda)}(y)=\sum_{j=1}^{N}\theta_j\Phi(y-y^j),\quad y\in\mathbb{R}^{d+1},
\end{equation}
with $\theta=\theta^{(n)}\in\mathbb{R}^N$ chosen subject to the native space norm constraint
$\theta^{\mathsf{T}}K\theta\le\lambda^2$ for some $\lambda>0$.

\subsection{Empirical loss and a posteriori bound}\label{sec:3.2}

Recall that $\mu_n$ and $\mu_n^{\rm T}$ are the normalized Lebesgue measures on $Q_n$ and
$B_{R_n}$, respectively.
For a coefficient vector $\theta\in\mathbb{R}^N$ define the population and empirical losses respectively by 
\begin{equation}
\label{eq:3.loss}
\begin{aligned}
 \mathcal L_n(\theta)
 &= \int_{Q_n}|r(v_n^{(\lambda)})(y)|^2\,d\mu_n
 + \int_{B_{R_n}}|v_n^{(\lambda)}(T,x)-g(x)|^2\,d\mu_n^{\rm T}, \\
 \hat{\mathcal L}_n(\theta)
 &= \frac{1}{N}\sum_{\ell=1}^{N}\bigl|\mathbf{r}^{\ell}(\theta)\bigr|^2
 + \frac{1}{|I_T|}\sum_{\ell\in I_T}\bigl|v_n^{(\lambda)}(T,\tilde x^\ell)-g(\tilde x^\ell)\bigr|^2,
\end{aligned}
\end{equation}
where $\mathbf{r}^{\ell}(\theta)=\sum_j\theta_j(\partial_t\Phi)(y^\ell-y^j)+H(y^\ell,D_xv_n^{(\lambda)}(y^\ell),D^2_{xx}v_n^{(\lambda)}(y^\ell))$
is the discrete residual at $y^\ell$, so that $\mathbf{r}^{\ell}(\theta)=r(v_n^{(\lambda)})(y^\ell)$, and $I_T$
indexes the $|I_T|=(N^x_n)^d$ terminal-slice collocation points.
The a posteriori residual bound
\begin{equation}
\label{eq:3.gamma}
 \gamma_n^{(\lambda)}:=\max_{\ell=1,\ldots,N}\bigl(|\mathbf{r}^{\ell}(\theta^{(n)})|\vee
 \mathbf{1}_{I_T}(\ell)|v_n^{(\lambda)}(T,\tilde x^\ell)-g(\tilde x^\ell)|\bigr)
\end{equation}
is computed after the optimization and reported in Section~\ref{sec:5}.

For practical solution we adopt the unconstrained $L^2$ penalty formulation
\begin{equation}
\label{eq:3.Pn-penalty}
 \min_{\theta\in\mathbb{R}^N}\;
 \hat{\mathcal L}_n(\theta)
   + \mu\,\bigl(\theta^{\mathsf{T}}K\theta-\lambda^2\bigr)_+^2,
\end{equation}
solved by the limited-memory BFGS (L-BFGS) method \cite{liu-noc:1989}. The initial point is the kernel interpolant of the
time-constant extension $(t,x)\mapsto g(x)$, rescaled if necessary to satisfy the norm constraint;
the penalty weight $\mu$ is chosen empirically,
and feasibility of the native space constraint $\theta^{\mathsf{T}}K\theta\le\lambda^2$ is checked
for each reported candidate (in the numerical experiments of Section~\ref{sec:5} we use
$\mu=10^2$). A finite penalty weight alone does not guarantee feasibility.

The empirical and population losses are related by the following elementary quadrature bound
for the collocation grid.

\begin{lem}[quadrature on the collocation grid]\label{lem:3.quad}
Assume $R_n\ge 1$, $N^t_n\ge1$ and $N^x_n\ge2$. There is a constant $C$ depending only on $d$ such that for every
Lipschitz function $f$ on $\overline{Q_n}$ and every Lipschitz function $u$ on $\overline{B_{R_n}}$,
\begin{equation}
\label{eq:3.quad}
\begin{aligned}
 \Bigl|\frac1N\sum_{\ell=1}^N f(y^\ell)-\int_{Q_n}f\,d\mu_n\Bigr|
   &\le C\,[f]_{C^{0,1}(\overline{Q_n})}\,R_nh_n, \\
 \Bigl|\frac1{|I_T|}\sum_{\ell\in I_T}u(\tilde x^\ell)-\int_{B_{R_n}}u\,d\mu^{\rm T}_n\Bigr|
   &\le C\,[u]_{C^{0,1}(\overline{B_{R_n}})}\,R_nh_n .
\end{aligned}
\end{equation}
\end{lem}
\begin{proof}
Index the grid points by $(k,j_1,\ldots,j_d)$ and set
\[
\begin{aligned}
 I_k&:=\left[\frac{kT}{N^t_n+1},\frac{(k+1)T}{N^t_n+1}\right),\qquad
 J_{j,i}:=\left[\bar x_i-R_n+\frac{2R_nj}{N^x_n},
                       \bar x_i-R_n+\frac{2R_n(j+1)}{N^x_n}\right),\\
 \mathcal C_{k,\mathbf j}&:=I_k\times\prod_{i=1}^dJ_{j_i,i}.
\end{aligned}
\]
Including the right endpoint when $k=N^t_n$ or $j=N^x_n-1$, these cells form a disjoint partition
of $\overline{Q_n}$. Each has volume $|Q_n|/N$, contains its associated grid point, and has side
lengths bounded by the corresponding grid spacings. Since
$h_n^2=(T/(2N^t_n))^2+d/(N^x_n-1)^2$, if $\mathcal C_\ell$ is associated with $y^\ell$, then
$\sup_{y\in\mathcal C_\ell}|y-y^\ell|\le2R_nh_n$. Therefore
\[
 \Bigl|\frac1N\sum_{\ell=1}^N f(y^\ell)-\int_{Q_n}f\,d\mu_n\Bigr|
 =\Bigl|\sum_{\ell=1}^N\frac{1}{|Q_n|}\int_{\mathcal C_\ell}\bigl(f(y^\ell)-f(y)\bigr)\,dy\Bigr|
 \le 2\,[f]_{C^{0,1}(\overline{Q_n})}\,R_nh_n .
\]
The spatial cells $\prod_{i=1}^dJ_{j_i,i}$ give the same estimate on the terminal slice.
\end{proof}

\subsection{Wendland kernel interpolation rate}\label{sec:3.3}

The convergence rate of the kernel realization is determined by the following standard
interpolation error estimate, which we record for the cylinder family $\{Q_n\}$ used here.
For $f\in\mathcal{H}$, let $I_n[f]$ denote its kernel interpolant on $\Gamma_n$, defined by
\[
 I_n[f](y)=\sum_{j=1}^N(K^{-1}\tilde f)^j\,\Phi(y-y^j),
 \qquad
 \tilde f=(f(y^1),\ldots,f(y^N))^{\mathsf{T}}.
\]

\begin{lem}\label{lem:3.kernel}
Assume $R_n\ge1$. There exist $h_0>0$ and $C>0$, independent of $n$ and $f$, such that for every
$f\in\mathcal H$, every multi-index $\alpha$ with $|\alpha|_1\le 2$, and every $n$ with
$R_nh_n\le h_0$,
\begin{equation}
\label{eq:3.kernel.error}
 \sup_{y\in\overline{Q_n}}|D^{\alpha}f(y)-D^{\alpha}I_n[f](y)|\le C\,(R_nh_n)^{\tau-(d+1)/2-|\alpha|_1}\,\|f\|_{\mathcal{H}},
\end{equation}
and $\|I_n[f]\|_{\mathcal{H}}\le\|f\|_{\mathcal{H}}$.
\end{lem}

The error estimate follows from \cite[Corollary~11.33]{wen:2005} on $Q_n$, with fill distance at most $R_nh_n$.
Uniformity in $n$ follows from the uniform interior-cone parameters of these boxes; see also
\cite[Lemma~3.6]{nak:2019a}. The final norm inequality follows from
\cite[Corollary~10.25]{wen:2005}.

\subsection{Qualitative convergence}\label{sec:3.4}

The following theorem establishes local uniform convergence of the kernel approximations.

\begin{thm}\label{thm:ker.qual}
Suppose that \textnormal{(A1)} and \textnormal{(A2)} hold, and that \eqref{eq:1.1} admits a classical solution $v$ that is the restriction to
$[0,T]\times\mathbb{R}^d$ of an element of $\mathcal{H}$, still denoted $v$.
Fix $\lambda>\|v\|_{\mathcal{H}}$ and let $\{\theta^{(n)}\}\subset\mathbb{R}^N$ be a sequence
such that
\begin{enumerate}
\item[\textnormal{(K1)}] $\theta^{(n)\mathsf{T}} K\theta^{(n)}\le\lambda^2$ for every $n$;
\item[\textnormal{(K2)}] $\hat{\mathcal L}_n(\theta^{(n)})\le \hat{\mathcal L}_n(\theta^{\rm int})+\varepsilon_n$,
where $\theta^{\rm int}=K^{-1}\tilde v$ is the kernel interpolation of $v$ and $\varepsilon_n\downarrow 0$.
\end{enumerate}
Suppose moreover that $R_n\to\infty$, $R_n^{d+1}h_n\to 0$ and $R_n^d\varepsilon_n\to0$.
Then $\sup_Q|v_n^{(\lambda)}-v|\to0$ for every compact
$Q\subset[0,T]\times\mathbb{R}^d$.
\end{thm}

By Lemma~\ref{lem:3.kernel}, $\theta^{\rm int}$ is feasible; hence \textnormal{(K2)} is not a numerical
test but an analytical comparison with the exact solution, satisfied by every feasible
$\varepsilon_n$-optimal solution.

\begin{proof}
Put $a:=\tau-(d+1)/2-2>1$. Since $H$ is Lipschitz in $(p,X)$ under
\textnormal{(A1)}, Lemma~\ref{lem:3.kernel} applied to $f=v$ gives, for the interpolant
$I_n[v]$, that is, the candidate \eqref{eq:3.kernel.ansatz} with coefficients $\theta^{\rm int}$,
$\sup_{Q_n}|r(I_n[v])|=\sup_{Q_n}|r(I_n[v])-r(v)|\le C\lambda(R_nh_n)^{a}$ and
$\sup_{B_{R_n}}|I_n[v](T,\cdot)-g|\le C\lambda(R_nh_n)^{a+2}$. Hence
$\hat{\mathcal L}_n(\theta^{\rm int})\le C\lambda^2(R_nh_n)^{2a}$, and by \textnormal{(K2)}
$\hat{\mathcal L}_n(\theta^{(n)})\le C\lambda^2(R_nh_n)^{2a}+\varepsilon_n$.
Also, \textnormal{(K1)}, the identity
$\|v_n^{(\lambda)}\|_{\mathcal H}^2=\theta^{(n)\mathsf T}K\theta^{(n)}$, and \eqref{eq:3.embed}
give a uniform $C^3_b$ bound. Lemmas~\ref{lem:2.F} and~\ref{lem:3.quad}, together with
\textnormal{(A2)} and the preceding loss estimate, then bound each normalized population $L^2$
error by $C(1+\lambda)^2R_nh_n+\varepsilon_n$ for large $n$. Since the domain volumes are
$O(R_n^d)$, the assumed limits and Lemma~\ref{lem:2.1} give
$\eta_n^{\rm res}\vee\eta_n^{\rm T}\to0$. Theorem~\ref{thm:tail}, applied with
$\lambda_n=C(1+\lambda)$, completes the proof because its exponential tail tends to zero.
\end{proof}

\subsection{Quantitative error bound}\label{sec:3.5}

The next theorem gives an explicit error bound in terms of the empirical loss, fill distance,
native space radius, and truncation radius.

\begin{thm}\label{thm:ker.HJB}
Suppose that \textnormal{(A1)} and \textnormal{(A2)} hold, let $v$ be the unique bounded continuous
viscosity solution of \eqref{eq:1.1}, and assume $R_n\ge1$ and $R_n\to\infty$.
For each $n$, let $\lambda_n>0$ and $\theta^{(n)}\in\mathbb{R}^{N_n}$ satisfy
$\theta^{(n)\mathsf T}K\theta^{(n)}\le\lambda_n^2$, and let $v_n^{(\lambda_n)}$ be the
corresponding kernel candidate in \eqref{eq:3.kernel.ansatz}.
Let $\hat{\mathcal L}_n^{\rm res}(\theta)$ and $\hat{\mathcal L}_n^{\rm trm}(\theta)$ denote the residual and
terminal parts of $\hat{\mathcal L}_n(\theta)$ in \eqref{eq:3.loss}. There is a constant
$C_{\rm quad}>0$, depending only on the fixed problem data and the kernel, such that, with
\[
 \mathcal E^{\rm res}_n=\hat{\mathcal L}_n^{\rm res}(\theta^{(n)})+C_{\rm quad}(1+\lambda_n)^2R_nh_n,
 \qquad
 \mathcal E^{\rm trm}_n=\hat{\mathcal L}_n^{\rm trm}(\theta^{(n)})+C_{\rm quad}(1+\lambda_n)^2R_nh_n,
\]
the following holds: for every compact $Q\subset[0,T]\times\mathbb{R}^d$ there exist
$C_{\rm err}>0$ and $n_0\in\mathbb{N}$ such that, for every $n\ge n_0$,
\begin{equation}
\label{eq:3.kernel.HJB}
\begin{aligned}
\sup_{Q}|v-v_n^{(\lambda_n)}|
 &\le C_{\rm err}\Bigl[
 R_n^{d/(d+3)}(1+\lambda_n)^{(d+1)/(d+3)}\bigl(\mathcal E^{\rm res}_n\bigr)^{1/(d+3)}
   +R_n^{d/2}\bigl(\mathcal E^{\rm res}_n\bigr)^{1/2} \\
 &\qquad +R_n^{d/(d+2)}(1+\lambda_n)^{d/(d+2)}\bigl(\mathcal E^{\rm trm}_n\bigr)^{1/(d+2)}
   +R_n^{d/2}\bigl(\mathcal E^{\rm trm}_n\bigr)^{1/2} \\
 &\qquad +(1+\lambda_n)\exp\bigl(-R_n^2/(8dC_1^2T)\bigr)\Bigr].
\end{aligned}
\end{equation}
\end{thm}
\begin{proof}
Step 1. The assumed constraint and the native space norm identity give
\[
 \|v_n^{(\lambda_n)}\|_{\mathcal H}^2
 =\theta^{(n)\mathsf T}K\theta^{(n)}\le\lambda_n^2.
\]
The Sobolev embedding \eqref{eq:3.embed} therefore gives
$\|v_n^{(\lambda_n)}\|_{C^{3}_b(\mathbb{R}^{d+1})}\le C\lambda_n$. Increasing $C\ge1$ if necessary,
$\|v_n^{(\lambda_n)}\|_{C^{1,2}([0,T]\times\mathbb{R}^d)}\le C(1+\lambda_n)$.
Lemma~\ref{lem:2.F} and
\textnormal{(A2)} give Lipschitz constants at most $C(1+\lambda_n)$ for $r(v_n^{(\lambda_n)})$ on
$\overline{Q_n}$ and for $v_n^{(\lambda_n)}(T,\cdot)-g$ on $\overline{B_{R_n}}$.

Step 2. Choose $C_{\rm quad}$ large enough for the two squared
Lipschitz estimates from Step~1 and apply Lemma~\ref{lem:3.quad} separately to the residual
and terminal parts. This gives
$\int_{Q_n}|r(v_n^{(\lambda_n)})|^2\,d\mu_n\le\mathcal E^{\rm res}_n$ and
$\int_{B_{R_n}}|v_n^{(\lambda_n)}(T,\cdot)-g|^2\,d\mu_n^{\rm T}\le\mathcal E^{\rm trm}_n$. Since
$|Q_n|,|B_{R_n}|\le CR_n^d$,
\[
 \|r(v_n^{(\lambda_n)})\|_{L^2(Q_n)}^2\le CR_n^d\,\mathcal E^{\rm res}_n,\qquad
 \|v_n^{(\lambda_n)}(T,\cdot)-g\|_{L^2(B_{R_n})}^2\le CR_n^d\,\mathcal E^{\rm trm}_n .
\]

Step 3. Lemma~\ref{lem:2.1} with $\alpha=1$, applied with $m=d+1$ on $Q_n$ and
$m=d$ on $B_{R_n}$ and combined with Steps~1 and~2, yields
\[
\begin{aligned}
 \sup_{Q_n}|r(v_n^{(\lambda_n)})|
 &\le C\,(R_n^d\mathcal E^{\rm res}_n)^{1/(d+3)}(1+\lambda_n)^{(d+1)/(d+3)}+C\,(R_n^d\mathcal E^{\rm res}_n)^{1/2},\\
 \sup_{B_{R_n}}|v_n^{(\lambda_n)}(T,\cdot)-g|
 &\le C\,(R_n^d\mathcal E^{\rm trm}_n)^{1/(d+2)}(1+\lambda_n)^{d/(d+2)}+C\,(R_n^d\mathcal E^{\rm trm}_n)^{1/2}.
\end{aligned}
\]

Step 4. By Step~1, $v_n^{(\lambda_n)}\in C^3_b(\mathbb{R}^{d+1})$ satisfies the hypotheses of
Theorem~\ref{thm:tail}, which applies with $\lambda_n$ replaced by
$C(1+\lambda_n)$. Use the sum of the two suprema in Step~3 as an upper bound for the $\eta_n$ of
\eqref{eq:tail.eta}, and choose $C_{\rm err}$ large enough to absorb the interpolation and
truncation constants. This gives \eqref{eq:3.kernel.HJB}.
\end{proof}

\begin{rem}\label{rem:ker.rate}
Estimate \eqref{eq:3.kernel.HJB} is a posteriori: it applies to every
$\theta^{(n)}$ satisfying the native space constraint, and its right-hand side
involves the computed empirical losses and prescribed parameters. If, moreover,
$v$ is a classical solution extending to an element of $\mathcal H$ with $\mathcal H$-norm less
than $\lambda_n$ and $\theta^{(n)}$ satisfies \textnormal{(K2)}, the argument in
the proof of Theorem~\ref{thm:ker.qual} gives
$\mathcal E_n^{\rm res}\vee\mathcal E_n^{\rm trm}
\le C(1+\lambda_n)^2R_nh_n+\varepsilon_n$
whenever $R_nh_n$ is sufficiently small. Consequently, the right-hand side of
\eqref{eq:3.kernel.HJB} tends to zero provided
$(1+\lambda_n)^{d+3}R_n^{d+1}h_n\to0$,
$(1+\lambda_n)^{d+1}R_n^d\varepsilon_n\to0$, and
$(1+\lambda_n)e^{-R_n^2/(8dC_1^2T)}\to0$.
When $\lambda_n\equiv\lambda$, these conditions and $R_n\to\infty$ are precisely
the assumptions of Theorem~\ref{thm:ker.qual}. The powers of $R_n$ enter when
converting the normalized losses into the unnormalized $L^2$ norms required by
Lemma~\ref{lem:2.1}.
\end{rem}

\section{Application to physics-informed neural networks}\label{sec:4}

We now apply the tools of Section~\ref{sec:2} to feedforward neural networks
using smooth activations. Throughout this section the activation function $\varrho$ is subject to
the following condition: 
\begin{enumerate}
\item[\textnormal{(S)}] $\varrho:\mathbb{R}\to\mathbb{R}$ is non-constant, $\varrho\in C^{\infty}(\mathbb{R})$,
$\varrho(0)=0$, $\|\varrho^{(j)}\|_{L^{\infty}(\mathbb{R})}<\infty$ for every $j\ge 0$, and
$\|\varrho'\|_{L^{\infty}(\mathbb{R})}\le 1$.
\end{enumerate}

\begin{rem}\label{rem:4.S}
The conditions $\varrho(0)=0$ and $\|\varrho'\|_{L^{\infty}}\le1$ are harmless normalizations;
only finitely many bounded derivatives are used. Condition \textnormal{(S)} includes $\tanh$ and
other bounded smooth activations. It excludes ReLU and unbounded smooth activations: for the latter
$v_{\theta}$ may grow at infinity, so the global bound
\eqref{eq:3.1.global} with $\alpha=0$, which enters the global $C^{1,2}$-bound required by
Theorem~\ref{thm:tail}, fails.
\end{rem}

\subsection{Network class and training problem}\label{sec:4.1}

Throughout this section the depth $L\in\mathbb{N}$ of the networks is fixed, while the width and
the weight bound may grow with $n$.
For $W\in\mathbb{N}$ and $B>0$, define the class of feedforward networks with activation $\varrho$,
of depth $L$, width at most $W$, and weight bound $B$ by
\begin{equation}
\label{eq:3.1.NN}
\begin{aligned}
 \mathcal V_{L,W,B}&:=\bigl\{v_{\theta}:[0,T]\times\mathbb{R}^d\to\mathbb{R}\,:\,
   v_{\theta}=A^{(L+1)}\varrho\circ(A^{(L)}\,\cdot+b^{(L)})\circ\cdots\circ\varrho\circ(A^{(1)}\,\cdot+b^{(1)})+b^{(L+1)},\\
 &\hspace*{4em}\theta=(A^{(\ell)},b^{(\ell)})_{\ell=1}^{L+1},\;
   A^{(\ell)}\in\mathbb{R}^{n_{\ell}\times n_{\ell-1}},\;
   b^{(\ell)}\in\mathbb{R}^{n_{\ell}},\;n_{\ell}\le W,\;
   |A^{(\ell)}|_{\infty}\vee|b^{(\ell)}|_{\infty}\le B\bigr\},
\end{aligned}
\end{equation}
where $\varrho$ is applied componentwise, $n_0=d+1$, and $n_{L+1}=1$, and $|\cdot|_{\infty}$ denotes
the maximum absolute value of the entries of a vector or matrix. Since $\varrho(0)=0$, a network with some $n_\ell<W$ coincides with a network whose hidden widths are
all equal to $W$, obtained by adding neurons with zero incoming and outgoing weights; we therefore take
$n_\ell=W$ for $1\le\ell\le L$. The resulting number of parameters is
\[
 P(L,W):=W(d+2)+(L-1)(W^2+W)+(W+1)
         =(L-1)W^2+(d+L+2)W+1.
\]
We write $\Theta_{L,W,B}=[-B,B]^{P(L,W)}$ for the corresponding parameter set. Since the class
\eqref{eq:3.1.NN} increases with $W$ and $B$, upper bounds proved for larger
values of these parameters apply to smaller ones; we therefore assume without loss of generality that
$W\ge d+1$ and $B\ge 1$, in which case $P(L,W)\le 2(L+1)W^2$.
By \textnormal{(S)}, every $v_{\theta}\in\mathcal V_{L,W,B}$ belongs to
$C^{\infty}([0,T]\times\mathbb{R}^d)$, and the following bounds hold.

\begin{lem}\label{lem:4.nn}
Assume \textnormal{(S)} and put $K:=WB$ and $P:=P(L,W)$. For every multi-index $\alpha$ there is a constant $C$,
depending only on $d$, $L$, $\varrho$ and $|\alpha|_1$, such that for all
$\theta,\theta'\in\Theta_{L,W,B}$,
\begin{equation}
\label{eq:3.1.global}
 \|D^{\alpha}v_{\theta}\|_{L^{\infty}([0,T]\times\mathbb{R}^d)}\le C\,K^{|\alpha|_1L+1},
\end{equation}
and, for every $R\ge 1$ and every $y\in\mathbb{R}^{d+1}$ with $|y|_{\infty}\le R$,
\begin{equation}
\label{eq:3.1.lip}
 |D^{\alpha}v_{\theta}(y)-D^{\alpha}v_{\theta'}(y)|
   \le C\,P\,R\,K^{(|\alpha|_1+1)(L+1)}\,|\theta-\theta'|_{\infty}.
\end{equation}
\end{lem}

\begin{proof}
Write $h^{(0)}(y)=y$, $z^{(\ell)}=A^{(\ell)}h^{(\ell-1)}+b^{(\ell)}$ and $h^{(\ell)}=\varrho(z^{(\ell)})$
for $1\le\ell\le L$, so that $v_{\theta}=A^{(L+1)}h^{(L)}+b^{(L+1)}$. Every row and every column of
every $A^{(\ell)}$ has $\ell^1$-norm at most $K$. We first show by induction on $\ell$ that, for $k\ge 1$,
\begin{equation}
\label{eq:3.1.ind}
 \sup_{y}\,\max_i\,\max_{|\beta|_1=k}\,|D^{\beta}h^{(\ell)}_i(y)|\le C\,K^{k\ell}.
\end{equation}
By the Fa\`a di Bruno formula, $D^{\beta}h^{(\ell)}_i$ is a sum, over the partitions $\pi$ of the
multi-index $\beta$, of terms $\varrho^{(|\pi|)}(z^{(\ell)}_i)\prod_{\gamma\in\pi}D^{\gamma}z^{(\ell)}_i$,
and $D^{\gamma}z^{(\ell)}_i=\sum_jA^{(\ell)}_{ij}D^{\gamma}h^{(\ell-1)}_j$. For $\ell=1$ only
$|\gamma|_1=1$ contributes, with $|D^{\gamma}z^{(1)}_i|\le B\le K$; for $\ell\ge 2$ the induction
hypothesis gives $|D^{\gamma}z^{(\ell)}_i|\le CK^{1+|\gamma|_1(\ell-1)}$. Since the derivatives of
$\varrho$ are bounded by \textnormal{(S)}, $|\pi|\le k$ and $K\ge1$, each summand corresponding to a
partition $\pi$ is bounded by $CK^{|\pi|+k(\ell-1)}\le CK^{k\ell}$. The number of such summands
depends only on $d$ and $k$, so it can be absorbed into $C$. This proves \eqref{eq:3.1.ind}. Consequently
$|D^{\alpha}v_{\theta}|\le K\cdot CK^{|\alpha|_1L}$ for $|\alpha|_1\ge1$, while
$|v_{\theta}|\le K\|\varrho\|_{L^{\infty}(\mathbb{R})}+B\le CK$, whence \eqref{eq:3.1.global} follows.

For $\vartheta\in\Theta_{L,W,B}$, let $\vartheta_p$, $1\le p\le P$, denote its $p$th scalar coordinate
after the entries of all weight matrices and bias vectors have been collected into a vector. For
\eqref{eq:3.1.lip}, the parameter set is convex, so the mean value theorem gives
$|D^{\alpha}v_{\theta}(y)-D^{\alpha}v_{\theta'}(y)|\le P\,|\theta-\theta'|_{\infty}\,
\max_p\sup_{\vartheta}|D^{\alpha}\partial_{\vartheta_p}v_{\vartheta}(y)|$, the supremum being taken over
$\vartheta\in\Theta_{L,W,B}$. Let $g^{(\ell)}:=\partial v_{\vartheta}/\partial z^{(\ell)}$, so that
$g^{(L)}=\varrho'(z^{(L)})\odot(A^{(L+1)})^{\mathsf T}$ and
$g^{(\ell-1)}=\varrho'(z^{(\ell-1)})\odot(A^{(\ell)})^{\mathsf T}g^{(\ell)}$, where $\odot$ denotes
componentwise multiplication. Then
$\partial v_{\vartheta}/\partial A^{(\ell)}_{ij}=g^{(\ell)}_ih^{(\ell-1)}_j$ and
$\partial v_{\vartheta}/\partial b^{(\ell)}_i=g^{(\ell)}_i$ for $\ell\le L$, while
$\partial v_{\vartheta}/\partial A^{(L+1)}_{1j}=h^{(L)}_j$ and $\partial v_{\vartheta}/\partial b^{(L+1)}=1$.
The argument leading to \eqref{eq:3.1.ind}, now applied to $\varrho'$, gives
$|D^{\gamma}[\varrho'(z^{(\ell)}_i)]|\le CK^{|\gamma|_1\ell}$. Hence
$|D^{\beta}g^{(L)}_i|\le CK^{1+|\beta|_1L}$. If the required bound holds for $g^{(\ell)}$, then the
backward recurrence above, the Leibniz rule, and the fact that every column of $A^{(\ell)}$ has
$\ell^1$-norm at most $K$ give
\[
 |D^{\beta}g^{(\ell-1)}_i|
 \le C\sum_{\gamma\le\beta}K^{|\gamma|_1(\ell-1)}K^{L+2-\ell}
          K^{(|\beta|_1-|\gamma|_1)L}
 \le CK^{L+2-\ell}K^{|\beta|_1L}.
\]
Thus, by backward induction,
$|D^{\beta}g^{(\ell)}_i|\le CK^{L+1-\ell}K^{|\beta|_1L}$. Moreover,
$|D^{\beta}h^{(\ell-1)}_j|\le CR K^{|\beta|_1(\ell-1)}$, where the factor $R$ accounts for
$h^{(0)}=y$. Hence $|D^{\alpha}\partial_{\vartheta_p}v_{\vartheta}(y)|\le CR K^{L+1+|\alpha|_1L}
\le CR K^{(|\alpha|_1+1)(L+1)}$ for every parameter $\vartheta_p$, and \eqref{eq:3.1.lip} follows.
\end{proof}

In particular, every $v_{\theta}\in\mathcal V_{L,W_n,B_n}$ belongs to $C^3_b([0,T]\times\mathbb{R}^d)$ and
has global $C^{1,2}$-norm at most $\Lambda_n:=C(W_nB_n)^{2L+1}$ with $C\ge1$, as required by
Theorem~\ref{thm:tail}.

For $\theta\in\Theta_{L,W_n,B_n}$ define the population and empirical losses
\begin{align*}
 \mathcal L_n(\theta)
  &= \int_{Q_n}|r(v_{\theta})|^2\,d\mu_n+\int_{B_{R_n}}|v_{\theta}(T,\cdot)-g|^2\,d\mu_n^{\rm T}, \\
 \hat{\mathcal L}_n(\theta)
  &= \frac{1}{M_n}\sum_{\ell=1}^{M_n}\left|r(v_{\theta})(\tilde y^{\ell})\right|^2
   + \frac{1}{M'_n}\sum_{\ell=1}^{M'_n}|v_{\theta}(T,\tilde x^{\ell})-g(\tilde x^{\ell})|^2,
\end{align*}
where $\tilde y^{1},\ldots,\tilde y^{M_n}$ are i.i.d.\ samples from $\mu_n$ and
$\tilde x^{1},\ldots,\tilde x^{M'_n}$ are i.i.d.\ samples from $\mu_n^{\rm T}$, the measures of
Section~\ref{sec:2.1}. For a fixed $s\in\mathbb{N}$ with $s>3+(d+1)/2$ and $\lambda_n>0$ we consider the
Sobolev-constrained problem
\begin{equation}
\tag{$\widehat P_n$}
\label{eq:Phatn}
 \text{minimize}\;\; \hat{\mathcal L}_n(\theta)
 \quad\text{subject to}\;\;  \theta\in\Theta_{L,W_n,B_n}, \;\; \|v_{\theta}\|_{H^s(Q_n)}\le \lambda_n .
\end{equation}
Given $\varepsilon_n>0$, an $\varepsilon_n$-optimal solution $\hat\theta_n$ of \eqref{eq:Phatn}
defines $v_n:=v_{\hat\theta_n}$.

\subsection{Approximation in Sobolev norm}\label{sec:4.2}

The key approximation result used below is that, on a compact set, a smooth function and all its
derivatives up to a prescribed order can be approximated uniformly by a single network with a
non-constant bounded smooth activation.

\begin{prop}[approximation in Sobolev norm]\label{prop:3.1}
Assume \textnormal{(S)}. Let $s\in\mathbb{N}$, $v\in H^{s}([0,T]\times\mathbb{R}^d)$, and let
$\{\delta_n\}$ be any sequence of positive numbers.
For each $n\in\mathbb{N}$ there exist $W_n\in\mathbb N$ with $W_n\ge d+1$, $B_n\ge1$ and a network
$v_{\theta^*_n}\in\mathcal V_{L,W_n,B_n}$ such that
\begin{equation}
\label{eq:3.2.approx}
 \|v-v_{\theta^*_n}\|_{H^s(Q_n)}\le \delta_n,\qquad
 \|v_{\theta^*_n}\|_{H^s(Q_n)}\le \|v\|_{H^s([0,T]\times\mathbb{R}^d)}+\delta_n. 
\end{equation}
\end{prop}
\begin{proof}
Fix $n\in\mathbb{N}$.

Step 1.
Since $Q_n$ is a bounded Lipschitz domain, choose $u_n\in C^{\infty}(\mathbb{R}^{d+1})$ such that
$\|v-u_n\|_{H^s(Q_n)}\le\delta_n/3$. Under \textnormal{(S)}, single-hidden-layer networks with
activation $\varrho$ approximate $u_n$ and its derivatives up to order $s$ uniformly on
$\overline{Q_n}$ \cite{hor-etal:1990,hor:1991}. Hence there is a depth-one network
$w_n(y)=a\cdot\varrho(Ay+b)+c$ such that
$\|u_n-w_n\|_{H^s(Q_n)}\le\delta_n/3$.

Step 2.
If $L\ge2$, choose $x_0\in\mathbb{R}$ with $\varrho'(x_0)\ne0$ and define
$\iota_{\kappa}(z):=(\varrho(x_0+\kappa z)-\varrho(x_0))/(\kappa\varrho'(x_0))$.
Taylor's theorem gives $\iota_{\kappa}\to\mathrm{id}$ in $C^s$ on compact sets as $\kappa\to0$.
We apply $\iota_{\kappa}$ componentwise and define
$w_{n,\kappa}(y):=a\cdot
\iota_{\kappa}^{\circ(L-1)}(\varrho(Ay+b))+c$.
Indeed, writing $\alpha_\kappa=1/(\kappa\varrho'(x_0))$ and
$\beta_\kappa=-\varrho(x_0)/(\kappa\varrho'(x_0))$, we have
$\iota_\kappa(z)=\alpha_\kappa\varrho(\kappa z+x_0)+\beta_\kappa$. Thus each application of
$\iota_\kappa$ adds one hidden layer: its outer affine map is combined with the next affine map, or
with the output layer for the last application. Since $w_n$ has one hidden layer, $w_{n,\kappa}$ has
depth $L$. The range of $\varrho(Ay+b)$ on $\overline{Q_n}$ is compact, and $L$ is fixed;
therefore $w_{n,\kappa}\to w_n$ in $C^s(\overline{Q_n})$, and hence in $H^s(Q_n)$. Choose
$\kappa$ so that $\|w_{n,\kappa}-w_n\|_{H^s(Q_n)}\le\delta_n/3$. If $L=1$, simply set
$w_{n,\kappa}:=w_n$.

Step 3.
Take $v_{\theta_n^*}:=w_{n,\kappa}$ and enlarge its width and weight bound, if necessary, so that
$W_n\ge d+1$ and $B_n\ge1$. The triangle inequality gives the first bound in
\eqref{eq:3.2.approx}. The second follows from that bound and
$\|v\|_{H^s(Q_n)}\le\|v\|_{H^s([0,T]\times\mathbb{R}^d)}$.
\end{proof}

\begin{rem}\label{rem:3.2.rate}
(a) Proposition~\ref{prop:3.1} is qualitative and gives no control of $W_n$ or $B_n$.
Theorem~\ref{thm:3.1} therefore imposes separate generalization and tail conditions involving these
parameters. The generalization term is polynomial in $W_n$ and logarithmic in $B_n$. The tail
condition is needed because the Sobolev constraint is local to $Q_n$, whereas
Theorem~\ref{thm:tail} requires the global $C^{1,2}$-bound
$C(W_nB_n)^{2L+1}$ supplied by Lemma~\ref{lem:4.nn}. In the kernel realization, the native space
constraint is global and its radius is fixed in Theorem~\ref{thm:ker.qual}, so no corresponding
parameter-growth condition is needed.

(b) For $\varrho=\tanh$ and $L=2$, suppose that
$v\in W^{s+k,\infty}((0,T)\times\mathbb{R}^d)$. After rescaling $Q_n$ to the
unit cube, De Ryck, Lanthaler and Mishra~\cite[Theorem~5.1]{drm:2021} give,
up to logarithmic factors,
\[
 \|v-v_{\theta}\|_{W^{s,\infty}(Q_n)}
 \le CR_n^pW_n^{-k/(d+1)},
\]
where $p$ depends only on $s,k,d$, and $B_n$ is polynomially bounded in $W_n$
and $R_n$; see also~\cite{dr-mis:2022}. Hence
the networks in Proposition~\ref{prop:3.1} can be chosen with
$\delta_n\le CR_n^{p+d/2}W_n^{-k/(d+1)}$. This rate exhibits the curse of
dimensionality, although sharper rates are available under additional structure. Any approximation
rate accompanied by compatible width and weight bounds can be inserted into \eqref{eq:3.5.Lstar}.
\end{rem}

Combining Proposition~\ref{prop:3.1} with the Sobolev embedding
$H^s(Q_n)\hookrightarrow C^{1,2}(\overline{Q_n})$ (valid for $s>2+(d+1)/2$,
in particular for $s>3+(d+1)/2$ as assumed) and the Lipschitz estimate of Lemma~\ref{lem:2.F}\,(ii),
the population loss at $\theta^*_n$ is controlled by the approximation error.

\begin{cor}\label{cor:3.1}
Under the assumptions of Proposition~$\ref{prop:3.1}$ with $s>3+(d+1)/2$, and if $v$ is a classical
solution of \eqref{eq:1.1}, we have
\begin{equation}
\label{eq:3.2.loss}
 \mathcal L_n(\theta^*_n)\le C\,\delta_n^2,
\end{equation}
where the constant $C$ depends on $C_1$ and the Sobolev embedding constant
of $H^s(Q_n)\hookrightarrow C^{1,2}(\overline{Q_n})$ (uniform in $n$ for the cylinder family).
\end{cor}
\begin{proof}
Since $r(v)\equiv 0$ classically, Lemma~\ref{lem:2.F}\,(ii) applied at a fixed point gives
\(
 |r(v_{\theta^*_n})|=|r(v_{\theta^*_n})-r(v)|\le C\,\|v-v_{\theta^*_n}\|_{C^{1,2}(\overline{Q_n})}.
\)
Sobolev embedding yields $\|v-v_{\theta^*_n}\|_{C^{1,2}(\overline{Q_n})}\le C\,\delta_n$,
and the residual contribution to $\mathcal L_n(\theta^*_n)$ is at most $C\,\delta_n^2$.
The terminal contribution is bounded in the same way, since $g=v(T,\cdot)$.
\end{proof}

\subsection{Generalization on the network class}\label{sec:4.3}

The next lemma converts an empirical $L^2$ bound into a population $L^2$ bound
via Rademacher-type concentration on a Sobolev ball.

For a class $\mathcal G$ of real-valued functions on a measurable set $\mathcal X$
and i.i.d.\ samples $Z_1,\ldots,Z_M$ from a probability measure on $\mathcal X$,
the (population) Rademacher complexity is
\begin{equation}
\label{eq:2.3.RM}
 \mathcal R_M(\mathcal G):=\mathbb{E}\Bigl[\sup_{\phi\in\mathcal G}
 \Bigl|\frac{2}{M}\sum_{i=1}^M\zeta_i\,\phi(Z_i)\Bigr|\Bigr],
\end{equation}
where $\zeta_1,\ldots,\zeta_M$ are i.i.d.\ Rademacher variables, that is,
$\mathbb{P}(\zeta_i=1)=\mathbb{P}(\zeta_i=-1)=1/2$, independent of $\{Z_i\}$.

\begin{lem}[Generalization on a Sobolev ball]\label{lem:2.3}
Let $\mathcal F_n=\{v_{\theta}: \theta\in\Theta_{L,W_n,B_n},\; \|v_{\theta}\|_{H^s(Q_n)}\le \lambda_n\}$, and define
\begin{align*}
 \mathcal F_n^{\rm r}&:=\{r(v_{\theta}):v_{\theta}\in\mathcal F_n\},\qquad
 \mathcal F_n^{\rm T}:=\{v_{\theta}(T,\cdot)-g:v_{\theta}\in\mathcal F_n\}.
\end{align*}
Suppose there exists $\Gamma_n>0$ such that $|\phi|\le \Gamma_n$ on $\overline{Q_n}$ for every $\phi\in\mathcal F_n^{\rm r}$
and $|\psi|\le \Gamma_n$ on $\overline{B_{R_n}}$ for every $\psi\in\mathcal F_n^{\rm T}$.
Then for any $\delta\in (0,1)$, with probability at least $1-\delta$ over the i.i.d.\ samples,
\begin{equation}
\label{eq:2.3.gen}
 \sup_{v_{\theta}\in\mathcal F_n}\bigl|\mathcal L_n(\theta)-\hat{\mathcal L}_n(\theta)\bigr|
   \le 4 \Gamma_n\bigl(\mathcal R_{M_n}(\mathcal F_n^{\rm r})+\mathcal R_{M'_n}(\mathcal F_n^{\rm T})\bigr)
    + \Gamma_n^2\,\sqrt{\frac{2\log(2/\delta)}{M_n\wedge M'_n}}.
\end{equation}
\end{lem}
\begin{proof}
For the residual part, set $\mathcal H_n^{\rm r}:=\{|\phi|^2:\phi\in\mathcal F_n^{\rm r}\}$ and let
$G_n^{\rm r}$ be the supremum over $h\in\mathcal H_n^{\rm r}$ of
$|\int_{Q_n}h\,d\mu_n-M_n^{-1}\sum_{\ell=1}^{M_n}h(\tilde y^\ell)|$. A ghost-sample argument
\cite[proof of Theorem~8]{bar-men:2002} and the contraction estimate
\cite[Theorem~12(4)]{bar-men:2002}, applied to a $2\Gamma_n$-Lipschitz extension of $u\mapsto u^2$
that vanishes at the origin, give
$\mathbb E[G_n^{\rm r}]\le\mathcal R_{M_n}(\mathcal H_n^{\rm r})
\le4\Gamma_n\mathcal R_{M_n}(\mathcal F_n^{\rm r})$.
If one sample $\tilde y^j$ is replaced while the others are fixed, then $0\le h\le\Gamma_n^2$ shows
that $G_n^{\rm r}$ changes by at most $\Gamma_n^2/M_n$. McDiarmid's inequality therefore implies,
with probability at least
$1-\delta/2$, $G_n^{\rm r}\le\mathbb E[G_n^{\rm r}]
+\Gamma_n^2\sqrt{\log(2/\delta)/(2M_n)}$.
Define $G_n^{\rm T}$ analogously. The same argument, a union bound, and
$\sup_{v_\theta\in\mathcal F_n}|\mathcal L_n(\theta)-\hat{\mathcal L}_n(\theta)|
\le G_n^{\rm r}+G_n^{\rm T}$, together with
\[
 \sqrt{\frac{\log(2/\delta)}{2M_n}}+\sqrt{\frac{\log(2/\delta)}{2M'_n}}
 \le \sqrt{\frac{2\log(2/\delta)}{M_n\wedge M'_n}},
\]
proves \eqref{eq:2.3.gen}.
\end{proof}

We bound the Rademacher complexities of the residual class $\mathcal F_n^{\rm r}$ and the terminal
class $\mathcal F_n^{\rm T}$ of Lemma~\ref{lem:2.3} by a covering argument in the parameters, which
uses only the regularity of $\theta\mapsto v_{\theta}$ from Lemma~\ref{lem:4.nn} and the Sobolev
constraint in \eqref{eq:Phatn}. Write $P_n:=P(L,W_n)\le 2(L+1)W_n^2$ for the number of parameters
of $\mathcal V_{L,W_n,B_n}$.

\begin{prop}\label{prop:3.2}
Assume \textnormal{(S)}, \textnormal{(A1)} and \textnormal{(A2)}, let
$s>3+(d+1)/2$, $W_n\ge d+1$, $B_n\ge1$, and let
$\mathcal F_n=\{v_{\theta}:\theta\in\Theta_{L,W_n,B_n},\;\|v_{\theta}\|_{H^s(Q_n)}\le\lambda_n\}$.
Suppose that the samples in \eqref{eq:2.3.RM} take values in $\overline{Q_n}$ for the residual class
and in $\overline{B_{R_n}}$ for the terminal class. Then there is a constant $C\ge1$, depending only on
$d$, $L$, $\varrho$, $T$, $\bar x$, $C_1$, $\|g\|_{L^{\infty}}$ and the constant of the Sobolev
embedding $H^s(Q_n)\hookrightarrow C^2(\overline{Q_n})$ (uniform in $n$), such that for every $M\ge 2$,
\begin{equation}
\label{eq:3.3.Rademacher}
\begin{aligned}
 \mathcal R_M(\mathcal F_n^{\rm r})+\mathcal R_M(\mathcal F_n^{\rm T})
   &\le C\,(1+\lambda_n)\sqrt{\frac{P_n\log\bigl(M(1+R_n)W_nB_n\bigr)}{M}} \\
   &\le C\,(1+\lambda_n)\,W_n\sqrt{\frac{\log\bigl(M(1+R_n)W_nB_n\bigr)}{M}}.
\end{aligned}
\end{equation}
\end{prop}
\begin{proof}
Throughout, $C$ denotes a constant with the dependence stated above, varying from line to line, and
$K_n:=W_nB_n\ge1$. We treat the residual class; the terminal class is handled in the same way.

Step 1. 
Set $r:=\min\{T,1\}$. Since $R_n\to\infty$, for all sufficiently large $n$ and every
$y\in\overline{Q_n}$ there is a translate $U_{n,y}$ of $(0,r)^{d+1}$ such that
$y\in\overline{U_{n,y}}\subset\overline{Q_n}$. Sobolev embedding on this fixed cube and translation
invariance give, for $|\alpha|_1\le2$,
\[
 |D^\alpha u(y)|\le C\|u\|_{H^s(U_{n,y})}\le C\|u\|_{H^s(Q_n)},
 \qquad s>2+(d+1)/2,
\]
where $C$ is independent of $n$ and $y$. The finitely many remaining $n$ can be absorbed by increasing
$C$. Since $s>3+(d+1)/2$ and $\|v_\theta\|_{H^s(Q_n)}\le\lambda_n$, it follows that
$\|v_{\theta}\|_{C^2(\overline{Q_n})}\le C\lambda_n$ for $v_{\theta}\in\mathcal F_n$.
By Lemma~\ref{lem:2.F}\,(i) and \textnormal{(A2)},
\begin{equation}
\label{eq:3.3.env}
 \sup_{\overline{Q_n}}|r(v_{\theta})|\le\Gamma_n,\qquad
 \sup_{\overline{B_{R_n}}}|v_{\theta}(T,\cdot)-g|\le\Gamma_n,\qquad
 \Gamma_n:=C(1+\lambda_n),\qquad v_{\theta}\in\mathcal F_n .
\end{equation}

Step 2.
Let $\bar R_n:=\max\{1,T,|\bar x|_{\infty}+R_n\}\le C(1+R_n)$, so that $|y|_{\infty}\le\bar R_n$ on
$\overline{Q_n}$. Lemma~\ref{lem:2.F}\,(ii), applied with the same $(t,x)$ in both arguments, gives
$|r(v_{\theta})(y)-r(v_{\theta'})(y)|\le|\partial_t(v_{\theta}-v_{\theta'})(y)|+C|D_x(v_{\theta}-v_{\theta'})(y)|
+C|D^2_{xx}(v_{\theta}-v_{\theta'})(y)|$, and \eqref{eq:3.1.lip} for $|\alpha|_1\le 2$ yields
\begin{equation}
\label{eq:3.3.lipres}
 |r(v_{\theta})(y)-r(v_{\theta'})(y)|\le\mathrm{Lip}_n\,|\theta-\theta'|_{\infty},\qquad
 \mathrm{Lip}_n:=C\,P_n\,\bar R_n\,K_n^{3L+3},
\end{equation}
for all $\theta,\theta'\in\Theta_{L,W_n,B_n}$ and $y\in\overline{Q_n}$. Note that $\theta'$ need not
satisfy the Sobolev constraint here.

Step 3.
Fix $z_1,\ldots,z_M\in\overline{Q_n}$ and $\epsilon>0$. The cube $[-B_n,B_n]^{P_n}$ is the union of at
most $N:=(1+B_n\mathrm{Lip}_n/\epsilon)^{P_n}$ cubes of side $2\epsilon/\mathrm{Lip}_n$. Choosing, in each
of them that meets the constrained set $\{\theta:v_{\theta}\in\mathcal F_n\}$, one parameter of that
set, \eqref{eq:3.3.lipres} provides $\phi_1,\ldots,\phi_{N'}\in\mathcal F_n^{\rm r}$, $N'\le N$, such that
every $\phi\in\mathcal F_n^{\rm r}$ satisfies $\max_i|\phi(z_i)-\phi_j(z_i)|\le2\epsilon$ for some $j$.
Hence
\[
 \mathbb{E}\sup_{\phi\in\mathcal F_n^{\rm r}}
 \left|\frac2M\sum_{i=1}^M\zeta_i\phi(z_i)\right|
 \le 4\epsilon+\frac2M\mathbb{E}\max_{j\le N'}
 \left|\sum_{i=1}^M\zeta_i\phi_j(z_i)\right|
 \le 4\epsilon+2\Gamma_n\sqrt{\frac{2\log(2N)}{M}}.
\]
To justify the last inequality, set $a_j:=(\phi_j(z_i))_{i\le M}\in\mathbb R^M$ and
$\mathcal A:=\{\pm a_j:j\le N'\}$. Then
$\max_j|\sum_i\zeta_i\phi_j(z_i)|=\max_{a\in\mathcal A}\sum_{i=1}^M\zeta_i a_i$,
$|\mathcal A|\le2N'\le2N$, and \eqref{eq:3.3.env} gives
$\max_{a\in\mathcal A}|a|\le\Gamma_n\sqrt M$. Massart's finite-class bound, obtained from
Hoeffding's exponential-moment estimate \cite{hoe:1963}, therefore gives
\[
 \mathbb E\max_{j\le N'}\left|\sum_{i=1}^M\zeta_i\phi_j(z_i)\right|
 \le\Gamma_n\sqrt{2M\log(2N')}\le\Gamma_n\sqrt{2M\log(2N)}.
\]
Multiplication by $2/M$ gives the last term in the preceding display.

Step 4.
Take $\epsilon=\Gamma_n/\sqrt M$. Since $\Gamma_n\ge1$, $B_n\le K_n$, $P_n\le 2(L+1)K_n^2$ and $M\ge2$,
\[
 \log(2N)\le \log 2+P_n\log\bigl(1+\sqrt M\,B_n\mathrm{Lip}_n\bigr)
 \le C\,P_n\log\bigl(M(1+R_n)K_n\bigr),
\]
so the Rademacher average in the first display of Step~3 is at most
$C\Gamma_n\sqrt{P_n\log(M(1+R_n)K_n)/M}$. This
bound does not depend on $z_1,\ldots,z_M$, so it persists after taking expectation over the samples.
For the terminal class one uses \eqref{eq:3.1.lip} with $\alpha=0$ at the points $(T,x)$,
$x\in\overline{B_{R_n}}$. The fixed function $g$ cancels in differences between parameter values and
is already included in the envelope bound \eqref{eq:3.3.env}. The second inequality in
\eqref{eq:3.3.Rademacher} follows from $P_n\le2(L+1)W_n^2$.
\end{proof}

\begin{rem}\label{rem:3.3.alt}
The dependence on $P_n$ in \eqref{eq:3.3.Rademacher} comes from covering the parameter cube, and the
same argument applies to the derivative classes entering the residual. Norm based bounds for network
outputs are available for Lipschitz activations vanishing at zero \cite{bar-fos-tel:2017} and, with
milder depth dependence, for positively homogeneous activations \cite{gol-rak-sha:2018}. These results
do not directly cover the derivative classes considered here, and positive homogeneity excludes
nonconstant bounded activations. The resulting bound is polynomial in $W_n$ and logarithmic in $B_n$
and $R_n$, while the fixed depth $L$ enters only through the constant.
\end{rem}

\subsection{Qualitative convergence}\label{sec:4.4}

The following theorem gives local uniform convergence in probability of the Sobolev-constrained network approximations.
\begin{thm}\label{thm:3.1}
Suppose that \textnormal{(S)}, \textnormal{(A1)} and \textnormal{(A2)} hold, and that \eqref{eq:1.1} has a classical solution
$v\in H^{s}([0,T]\times\mathbb{R}^d)$ for some integer $s>3+(d+1)/2$.
Fix $\lambda>\|v\|_{H^s([0,T]\times\mathbb{R}^d)}$, let $\delta_n>0$ satisfy $\delta_n\to0$ and let $W_n,B_n$ be as in
Proposition~$\ref{prop:3.1}$. For each $n$, let $v_n:=v_{\hat\theta_n}$ be $\varepsilon_n$-optimal solutions of
\eqref{eq:Phatn} on $\mathcal V_{L,W_n,B_n}$ with $\lambda_n=\lambda$. Suppose that $R_n\to\infty$,
$m_n:=M_n\wedge M'_n\to\infty$ and
\begin{equation}
\label{eq:3.4.scal}
 R_n^d\Bigl(\delta_n^2+W_n\sqrt{\frac{\log\bigl(m_n(1+R_n)W_nB_n\bigr)}{m_n}}+\varepsilon_n\Bigr)\to0,
 \qquad (W_nB_n)^{2L+1}e^{-R_n^2/(8dC_1^2T)}\to0 .
\end{equation}
Then $\sup_{Q}|v_n-v|\to0$ in probability for every compact $Q\subset[0,T]\times\mathbb{R}^d$.
\end{thm}
\begin{proof}
Step 1. Since $s>3+(d+1)/2$, the fixed-cube argument in the proof of
Proposition~\ref{prop:3.2} gives a uniform embedding
$H^s(Q_n)\hookrightarrow C^3(\overline{Q_n})$. The constraint in \eqref{eq:Phatn} therefore gives
$\|v_n\|_{C^3(\overline{Q_n})}\le C\lambda$. Lemma~\ref{lem:2.F}\,(iii), applied on the convex set
$Q_n$, and \textnormal{(A2)} then give
\[
 [r(v_n)]_{C^{0,1}(\overline{Q_n})}
 +[v_n(T,\cdot)-g]_{C^{0,1}(\overline{B_{R_n}})}\le C,
\]
where $C$ is independent of $n$.

Step 2. For large $n$, Proposition~\ref{prop:3.1} makes $v_{\theta_n^*}$ feasible because
$\|v_{\theta_n^*}\|_{H^s(Q_n)}
 \le\|v\|_{H^s([0,T]\times\mathbb R^d)}+\delta_n<\lambda$.
Let $\mathcal F_n$ be the constrained network class with Sobolev radius $\lambda$, and set
\[
 G_n:=\sup_{v_\theta\in\mathcal F_n}
       |\mathcal L_n(\theta)-\hat{\mathcal L}_n(\theta)|.
\]
Since $\hat\theta_n$ is $\varepsilon_n$-optimal for the empirical loss, the definition of $G_n$ and
Corollary~\ref{cor:3.1} give
\[
 \mathcal L_n(\hat\theta_n)
 \le\hat{\mathcal L}_n(\hat\theta_n)+G_n
 \le\hat{\mathcal L}_n(\theta_n^*)+\varepsilon_n+G_n
 \le \mathcal L_n(\theta_n^*)+2G_n+\varepsilon_n
 \le C\delta_n^2+2G_n+\varepsilon_n.
\]
Since $\lambda$ is fixed, Proposition~\ref{prop:3.2} gives the envelope
$\Gamma_n=C(1+\lambda)=O(1)$ required in Lemma~\ref{lem:2.3}. Apply Lemma~\ref{lem:2.3}
and Proposition~\ref{prop:3.2} with $M=M_n$ and $M=M'_n$. Since
$x\mapsto\log(a_nx)/x$, with $a_n:=(1+R_n)W_nB_n$, is decreasing when $a_nx\ge e$,
both resulting bounds can be written in terms of
$m_n=M_n\wedge M'_n$. Thus, for any fixed $\rho\in(0,1)$, with probability at least $1-\rho$,
\[
 G_n\le C W_n\sqrt{\frac{\log(m_n(1+R_n)W_nB_n)}{m_n}}
 +C\sqrt{\frac{\log(2/\rho)}{m_n}}.
\]
For fixed $\rho$, the last term is absorbed into the first because $W_n\ge d+1$ and
$\log(m_n(1+R_n)W_nB_n)\ge\log m_n$. Since $\rho$ is arbitrary,
$R_n^d\mathcal L_n(\hat\theta_n)\to0$ in probability by \eqref{eq:3.4.scal}.

Step 3. Since $\mu_n$ and $\mu_n^{\rm T}$ are normalized Lebesgue measures and
$|Q_n|,|B_{R_n}|\le CR_n^d$,
\[
 \|r(v_n)\|_{L^2(Q_n)}^2\le CR_n^d\mathcal L_n(\hat\theta_n),\qquad
 \|v_n(T,\cdot)-g\|_{L^2(B_{R_n})}^2\le CR_n^d\mathcal L_n(\hat\theta_n).
\]
Lemma~\ref{lem:2.1} with $\alpha=1$, applied in dimensions $d+1$ and $d$, respectively, and the
Lipschitz bounds of Step~1 yield
\[
\begin{aligned}
 \sup_{Q_n}|r(v_n)|
 &\le C\bigl(R_n^d\mathcal L_n(\hat\theta_n)\bigr)^{1/(d+3)}
      +C\bigl(R_n^d\mathcal L_n(\hat\theta_n)\bigr)^{1/2},\\
 \sup_{B_{R_n}}|v_n(T,\cdot)-g|
 &\le C\bigl(R_n^d\mathcal L_n(\hat\theta_n)\bigr)^{1/(d+2)}
      +C\bigl(R_n^d\mathcal L_n(\hat\theta_n)\bigr)^{1/2}.
\end{aligned}
\]
It follows that
$\eta_n:=\sup_{Q_n}|r(v_n)|+\sup_{B_{R_n}}|v_n(T,\cdot)-g|\to0$ in probability.

Step 4. Under \textnormal{(A1)} and \textnormal{(A2)}, the classical solution $v$ coincides with the unique
bounded continuous viscosity solution in Theorem~\ref{thm:tail}. Moreover, Lemma~\ref{lem:4.nn} gives
$\|v_n\|_{C^{1,2}([0,T]\times\mathbb R^d)}\le C(W_nB_n)^{2L+1}$.
Theorem~\ref{thm:tail}, the convergence of $\eta_n$ and the second condition in
\eqref{eq:3.4.scal} now prove the claim.
\end{proof}

\subsection{Quantitative error bound}\label{sec:4.5}

The following theorem makes the preceding convergence result quantitative by separating the contributing error terms.
\begin{thm}\label{thm:3.2}
Suppose that \textnormal{(S)}, \textnormal{(A1)} and \textnormal{(A2)} hold and that \eqref{eq:1.1}
admits a classical solution $v$ satisfying
$v\in H^{s}([0,T]\times\mathbb{R}^d)$ for some integer $s>3+(d+1)/2$.
Let $\delta_n>0$, let $W_n\ge d+1$, $B_n\ge1$ and $\theta^*_n\in\Theta_{L,W_n,B_n}$ satisfy
\eqref{eq:3.2.approx}, and let $R_n\ge 1$, $R_n\to\infty$, $M_n,M'_n\ge2$ and
$\lambda_n\ge\|v\|_{H^s([0,T]\times\mathbb{R}^d)}+\delta_n$.
For each $n$, let $v_n:=v_{\hat\theta_n}$ be $\varepsilon_n$-optimal solutions of \eqref{eq:Phatn}, and put
$m_n:=M_n\wedge M'_n$.
For every compact $Q\subset[0,T]\times\mathbb{R}^d$ there are $C_Q>0$ and $n_0\in\mathbb{N}$,
independent of $n$ and $\rho$, such that, for every $n\ge n_0$ and $\rho\in(0,1)$, with probability
at least $1-\rho$,
\begin{equation}
\label{eq:3.5.HJBrate}
\begin{aligned}
\sup_{(t,x)\in Q}|v(t,x)-v_n(t,x)|\le C_Q\Bigl[&
 R_n^{d/(d+3)}(1+\lambda_n)^{(d+1)/(d+3)}\bigl(\mathcal L^{*}_n(\rho)\bigr)^{1/(d+3)}\\
&+R_n^{d/2}\bigl(\mathcal L^{*}_n(\rho)\bigr)^{1/2}
 +R_n^{d/(d+2)}(1+\lambda_n)^{d/(d+2)}\bigl(\mathcal L^{*}_n(\rho)\bigr)^{1/(d+2)}\\
&+(W_nB_n)^{2L+1}e^{-R_n^2/(8dC_1^2T)}\Bigr],
\end{aligned}
\end{equation}
where
\begin{equation}
\label{eq:3.5.Lstar}
 \mathcal L^{*}_n(\rho):= \delta_n^2
   +(1+\lambda_n)^2W_n\sqrt{\frac{\log\bigl(m_n(1+R_n)W_nB_n\bigr)}{m_n}}
   +(1+\lambda_n)^2\sqrt{\frac{\log(2/\rho)}{m_n}}+\varepsilon_n .
\end{equation}
\end{thm}
\begin{proof}
Step 1. The regularity argument in Step~1 of the proof of Theorem~\ref{thm:3.1}, with $\lambda$ replaced
by $\lambda_n$, gives
$[r(v_n)]_{C^{0,1}(\overline{Q_n})}+[v_n(T,\cdot)-g]_{C^{0,1}(\overline{B_{R_n}})}\le C(1+\lambda_n)$.

Step 2. The assumption on $\lambda_n$ and \eqref{eq:3.2.approx} make $v_{\theta_n^*}$ feasible.
The comparison and generalization arguments in Step~2 of the proof of Theorem~\ref{thm:3.1}, with the
envelope $\Gamma_n=C(1+\lambda_n)$, then give, with probability at least $1-\rho$,
\[
 \mathcal L_n(\hat\theta_n)\le C\mathcal L^{*}_n(\rho).
\]

Step 3. The $L^2$-to-$L^\infty$ argument in Step~3 of the proof of Theorem~\ref{thm:3.1}, now keeping
the dependence on $\lambda_n$ and $\mathcal L^{*}_n(\rho)$, gives on the same event
\begin{equation}
\label{eq:3.5.eta}
\begin{aligned}
 \eta_n:=\sup_{Q_n}|r(v_n)|+\sup_{B_{R_n}}|v_n(T,\cdot)-g|
 &\le C R_n^{d/(d+3)}(1+\lambda_n)^{(d+1)/(d+3)}(\mathcal L^{*}_n(\rho))^{1/(d+3)}\\
 &\quad+C R_n^{d/2}(\mathcal L^{*}_n(\rho))^{1/2}
 +C R_n^{d/(d+2)}(1+\lambda_n)^{d/(d+2)}(\mathcal L^{*}_n(\rho))^{1/(d+2)}.
\end{aligned}
\end{equation}

Step 4. The truncation argument in Step~4 of the proof of Theorem~\ref{thm:3.1}, with
$\|v_n\|_{C^{1,2}([0,T]\times\mathbb{R}^d)}\le C(W_nB_n)^{2L+1}$ from Lemma~\ref{lem:4.nn},
combines \eqref{eq:3.5.eta} with Theorem~\ref{thm:tail} and gives \eqref{eq:3.5.HJBrate}.
\end{proof}

\begin{rem}\label{rem:3.5.balance}
For every admissible choice of the parameters, estimate \eqref{eq:3.5.HJBrate} holds with the stated
probability. The four terms in \eqref{eq:3.5.Lstar} are, respectively, the approximation,
generalization, concentration and optimization errors. For any fixed $\rho\in(0,1)$, the right-hand
side of \eqref{eq:3.5.HJBrate} tends to zero if
\begin{equation}
\label{eq:3.5.balance}
 (1+\lambda_n)^{d+1}R_n^d\,\mathcal L^{*}_n(\rho)\to 0
 \qquad\text{and}\qquad
 (W_nB_n)^{2L+1}e^{-R_n^2/(8dC_1^2T)}\to 0,
\end{equation}
These conditions are the quantitative counterpart of \eqref{eq:3.4.scal}. Explicit rates follow once
an approximation rate is available. Suppose that $\delta_n\le CR_n^pW_n^{-\kappa}$ and
$B_n\le CR_n^pW_n^p$ for some $\kappa,p>0$, and take $\lambda_n=\lambda$ fixed, $M'_n=M_n$,
$\varepsilon_n\le M_n^{-\kappa/(2\kappa+1)}$,
$W_n\simeq M_n^{1/(2(2\kappa+1))}$ and $R_n=A(\log M_n)^{1/2}$ with $A$ large. Then both conditions in
\eqref{eq:3.5.balance} hold, and \eqref{eq:3.5.HJBrate} gives
$\sup_Q|v-v_n|=O_{\mathbb P}(M_n^{-\beta})$ up to logarithmic factors, where
$\beta=\kappa/((2\kappa+1)(d+3))$. For
$\tanh$ with $L=2$ and $v\in W^{s+k,\infty}$, Remark~\ref{rem:3.2.rate}\,(b) gives $\kappa=k/(d+1)$.
The factors that weaken \eqref{eq:3.5.HJBrate} have distinct origins. The powers of $R_n$ arise from
the volumes of the truncated domains, the powers of $1+\lambda_n$ arise from the
$L^2$-to-$L^{\infty}$ conversion, and the exponential term accounts for truncation of the unbounded
state space.
\end{rem}

\section{Numerical experiments}\label{sec:5}

We illustrate the kernel and network realizations developed in Sections~\ref{sec:3}
and~\ref{sec:4} on a manufactured test problem, with emphasis on refinement, sampling and
regularization effects.

\subsection{Test problem}\label{sec:5.1}

We use an equation of the form \eqref{eq:2.HJB} with a manufactured source. Let $T=1$,
$A=[\sigma_{\min},\sigma_{\max}]\times[-\bar b,\bar b]$ with
$(\sigma_{\min},\sigma_{\max},\bar b)=(0.1,0.4,1)$, and for $a=(\varsigma,\nu)\in A$ let
$\sigma(t,x,a)=\varsigma I_d$, $b(t,x,a)=\nu\,\mathbf 1/d$ with $\mathbf 1=(1,\ldots,1)$, and $f(t,x,a)=f(t,x)$. Then
$\inf_{a\in A}H^a(t,x,p,X)=\tfrac12\sigma_{\min}^2(\mathrm{tr}X)^+-\tfrac12\sigma_{\max}^2(\mathrm{tr}X)^-
-\tfrac{\bar b}{d}\bigl|\sum_ip_i\bigr|+f(t,x)$, and \eqref{eq:1.1} reads
\begin{equation}
\label{eq:4.testHJB}
\left\{
\begin{aligned}
 &\partial_t v+\tfrac12\sigma_{\min}^2(\Delta v)^+-\tfrac12\sigma_{\max}^2(\Delta v)^-
   -\tfrac{\bar b}{d}\Bigl|\textstyle\sum_{i=1}^d\partial_{x_i}v\Bigr|+f=0,\quad (t,x)\in[0,1)\times\mathbb{R}^d, \\
 & v(1,x)=g(x),\quad x\in\mathbb{R}^d.
\end{aligned}
\right.
\end{equation}
We choose $f$ and $g$ so that the solution is
\[
 v(t,x)=\sin\Bigl(t+\textstyle\sum_{i=1}^d x_i\Bigr)\,e^{-|x|^2/\ell_0^2},\qquad \ell_0=2,
\]
a spatially localized version of the benchmark solution $\sin(t+\sum_ix_i)$ of Guo et
al.~\cite{guo-etal:2015}: $g=v(1,\cdot)$ and $f$ is obtained by inserting $v$ into
\eqref{eq:4.testHJB}. All hypotheses of Sections~\ref{sec:3} and~\ref{sec:4} hold: $A$ is compact,
$\sigma$ and $b$ are constant, $f$ and $g$ are bounded and Lipschitz,
$\sigma\sigma^{\mathsf T}\ge\sigma_{\min}^2I_d$, and, owing
to the Gaussian factor, $v\in H^{s}([0,T]\times\mathbb{R}^d)$ for all $s$ and $v$ extends to an
element of $H^{\tau}(\mathbb{R}^{d+1})$. The nonlinearity is active: both $\Delta v$ and
$\sum_i\partial_{x_i}v$ change sign inside the core $[0,1]\times[-1,1]^d$ used below, so the optimal
diffusion level and drift switch there. On the boundary of the truncation cylinders with $R=3$ used
below, $|v|\le e^{-9/4}\approx0.1$, so the artificial boundary is not negligible. The residual used by
both realizations is the left-hand side of \eqref{eq:4.testHJB} with $v$ replaced by $\varphi$.

\subsection{Kernel realization: implementation and results}\label{sec:5.2}

\smallskip
\noindent\textbf{Choice of the kernel.}
To meet the requirement $\tau>3+(d+1)/2$ in Section~\ref{sec:3.1}, we take
$\Phi(y)=\phi(|y|)$ with the rescaled Wendland function
\begin{equation}
\label{eq:5.wendland}
 \phi(r)=\phi_{3,3}(r/c),\qquad
 \phi_{3,3}(s)=(1-s)_+^{8}\bigl(32s^3+25s^2+8s+1\bigr),
\end{equation}
parameterised by $c>0$. The function $\phi_{3,3}$ is $C^6$ and positive definite on
$\mathbb{R}^m$ for $m\le3$, and its native space on $\mathbb{R}^3$ is $H^5(\mathbb{R}^3)$
\cite[Chapter~9]{wen:2005}. Its Fourier decay gives $\tau=9/2$ on $\mathbb{R}^2$ and $\tau=5$ on
$\mathbb{R}^3$, satisfying \eqref{eq:3.tau} for $d=1$ and $d=2$.

\smallskip
\noindent\textbf{Discretization and optimization.}
The collocation set $\Gamma_n\subset\overline{Q_n}=[0,1]\times[-R,R]^d$ is a uniform tensor-product grid
with $N^t+1$ time nodes $t_k=k/N^t$ and $N^x$ spatial nodes per axis on $[-R,R]$,
giving $N=(N^t+1)(N^x)^d$ points; an additional set $\Gamma^{\mathrm{extra}}$ of $N^{\mathrm{extra}}$
Sobol' points in $[0,1-10^{-8})\times[-R,R]^d$ and corresponding terminal points are added during
training. The theoretical loss \eqref{eq:3.loss} uses only the tensor grid. In the spatial-refinement
studies, we fix $N^t=32$ for $d=1$ and $N^t=16$ for $d=2$, refine $N^x$, and report
$h_x=2R/(N^x-1)$. The full space-time fill distance is
$h_n=[(T/(2N^t))^2+d/(N^x-1)^2]^{1/2}$ and therefore has a positive floor in those studies. Their
slopes against $h_x$ are visual references only. A separate experiment below grows $R$ and refines
both $N^t$ and $N^x$ according to the scaling in Theorem~\ref{thm:ker.qual}.
In these experiments, we use $\mu=10^2$, at most $2\times10^4$ L-BFGS iterations, and
$\lambda=6$. This value exceeds all computed interpolation norms (at most $4.85$)
but does not certify the strict inequality in Theorem~\ref{thm:ker.qual}; all computed minimisers
remain below $\lambda$. CPU times are wall-clock times on an Apple~M3~Pro with multi-threaded BLAS.

\smallskip
\noindent\textbf{Error metrics.}
We evaluate the fixed core $Q^{\mathrm{core}}=[0,1]\times[-1,1]^d$ on an independent grid of
$M^{\mathrm{core}}$ points ($21\times41$ for $d=1$ and $21\times21^2$ for $d=2$), and define
\[
 E^{\mathrm{core}}_{\infty}:=\max_{\ell}|v^{(\lambda)}(\tilde y^{\ell})_n-v(\tilde y^{\ell})|,\qquad
 E^{\mathrm{core}}_{\mathrm{RMS}}:=\Bigl(\tfrac{1}{M^{\mathrm{core}}}\textstyle\sum_{\ell}|v^{(\lambda)}_n(\tilde y^{\ell})-v(\tilde y^{\ell})|^2\Bigr)^{1/2},
 \quad \tilde y^{\ell}\in Q^{\mathrm{core}};
\]
The quantities $E_{\mathrm{RMS}}$ and $E_{\infty}$ denote the corresponding errors on the full
cylinder, evaluated on $51\times101$ points for $d=1$ and $21\times41^2$ points for $d=2$.
We compute $\gamma_n$ in \eqref{eq:3.gamma} at the collocation points.

\smallskip
\noindent\textbf{Spatial-grid refinement ($d=1$).}
We fix $R=3$ and $c=1$ and refine $N^x$ from $9$ to $49$, i.e.\ $h_x$ from $0.75$ to
$0.125$. Table~\ref{tab:ker.refine_d1} and Figure~\ref{fig:ker.refine} report the outcome.

\begin{table}[tb]
  \centering
  \small
  \caption{Spatial-grid refinement of the kernel realization ($d=1$).}
  \label{tab:ker.refine_d1}
  \begin{tabular}{cccccccccc}
    \toprule
    $N^t$ & $N^x$ & $N$ & $h_x$ & $\gamma_n$ & $\|v_n\|_{\mathcal H}$ & $E_{\mathrm{RMS}}^{\mathrm{core}}$ & $E_{\infty}^{\mathrm{core}}$ & $E_{\mathrm{RMS}}$ & CPU [s] \\
    \midrule
    32 & 9 & 297 & 0.7500 & 3.733e-01 & 1.96 & 3.467e-01 & 6.352e-01 & 2.902e-01 & 0.5 \\
    32 & 13 & 429 & 0.5000 & 1.428e-01 & 2.48 & 1.492e-01 & 3.170e-01 & 1.126e-01 & 5.0 \\
    32 & 17 & 561 & 0.3750 & 4.146e-02 & 2.75 & 1.875e-02 & 3.885e-02 & 1.273e-02 & 14.0 \\
    32 & 25 & 825 & 0.2500 & 4.841e-02 & 2.78 & 1.513e-03 & 3.856e-03 & 5.677e-03 & 32.1 \\
    32 & 33 & 1089 & 0.1875 & 7.663e-02 & 2.78 & 9.977e-04 & 5.007e-03 & 4.265e-03 & 21.2 \\
    32 & 49 & 1617 & 0.1250 & 7.788e-02 & 2.78 & 7.735e-04 & 4.972e-03 & 1.574e-03 & 46.3 \\
    \bottomrule
  \end{tabular}
\end{table}

\begin{figure}[tb]
\centering
\includegraphics[width=0.98\columnwidth]{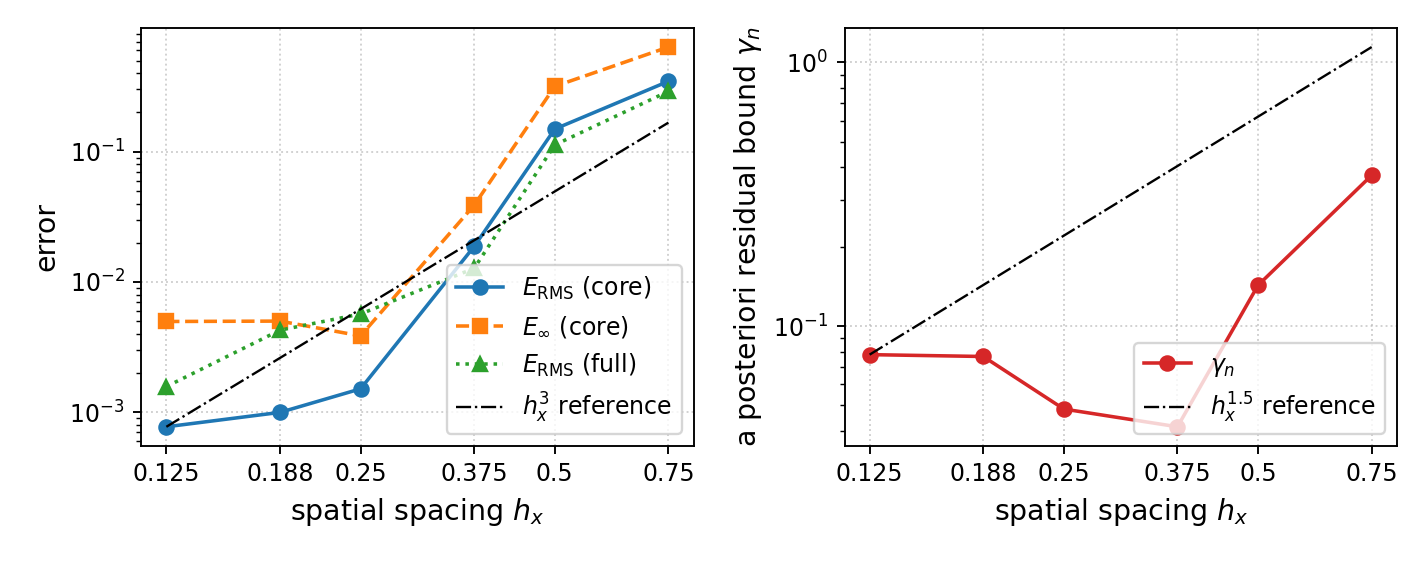}
\caption{Kernel realization, $d=1$, $R=3$, $c=1$, $\lambda=6$.
(left) Errors on the fixed compact core $Q^{\mathrm{core}}=[0,1]\times[-1,1]$ and on the full
truncation cylinder $[0,1]\times[-3,3]$, versus the spatial spacing $h_x$, with a reference slope
$h_x^3$.
(right) The a posteriori residual bound $\gamma_n$ of \eqref{eq:3.gamma}, versus $h_x$,
against an $h_x^{3/2}$ reference slope.}
\label{fig:ker.refine}
\end{figure}

The errors decrease rapidly and the core RMS error falls below $10^{-3}$ on the two finest grids.
By contrast, $\gamma_n$ reaches its smallest value at $N^x=17$ and then rises, indicating an
optimization floor.

\smallskip
\noindent\textbf{Bandwidth and error localization ($d=1$).}
The left panel of Figure~\ref{fig:ker.sens} varies the bandwidth $c$ at $(N^t,N^x)=(32,25)$, $R=3$
(Table~\ref{tab:ker.eps_d1}); the right panel shows the error profile
$x\mapsto\max_{t}|v^{(\lambda)}_n-v|(t,x)$ for two bandwidths.

\begin{table}[tb]
  \centering
  \small
  \caption{Bandwidth sensitivity of the kernel realization ($d=1$).}
  \label{tab:ker.eps_d1}
  \begin{tabular}{cccccc}
    \toprule
    $c$ & $\gamma_n$ & $\|v_n\|_{\mathcal H}$ & $E_{\mathrm{RMS}}^{\mathrm{core}}$ & $E_{\infty}^{\mathrm{core}}$ & CPU [s] \\
    \midrule
    0.5 & 7.331e-01 & 3.04 & 3.259e-01 & 6.173e-01 & 1.2 \\
    0.75 & 7.097e-02 & 3.40 & 1.541e-02 & 3.934e-02 & 14.3 \\
    1 & 4.841e-02 & 2.78 & 1.513e-03 & 3.856e-03 & 32.4 \\
    1.25 & 4.220e-02 & 2.39 & 6.111e-04 & 2.009e-03 & 30.7 \\
    1.5 & 5.825e-02 & 2.15 & 9.825e-04 & 2.816e-03 & 17.2 \\
    2 & 5.101e-02 & 1.82 & 8.268e-04 & 2.765e-03 & 17.7 \\
    2.5 & 4.232e-02 & 1.68 & 9.428e-04 & 2.508e-03 & 15.4 \\
    3 & 3.980e-02 & 1.58 & 9.999e-04 & 2.248e-03 & 23.1 \\
    4 & 4.421e-02 & 1.62 & 1.264e-03 & 3.487e-03 & 19.9 \\
    \bottomrule
  \end{tabular}
\end{table}

\begin{figure}[tb]
\centering
\begin{minipage}[b]{0.49\columnwidth}
  \centering
  \includegraphics[width=0.99\linewidth]{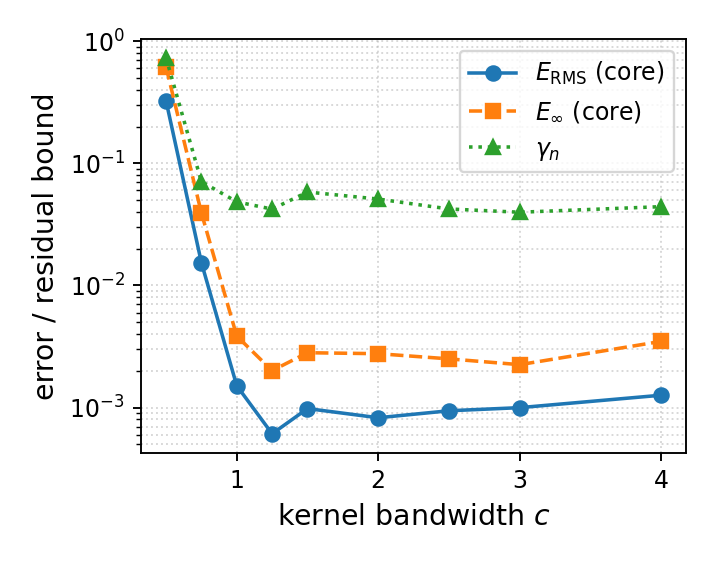}
\end{minipage}\hfill
\begin{minipage}[b]{0.49\columnwidth}
  \centering
  \includegraphics[width=0.99\linewidth]{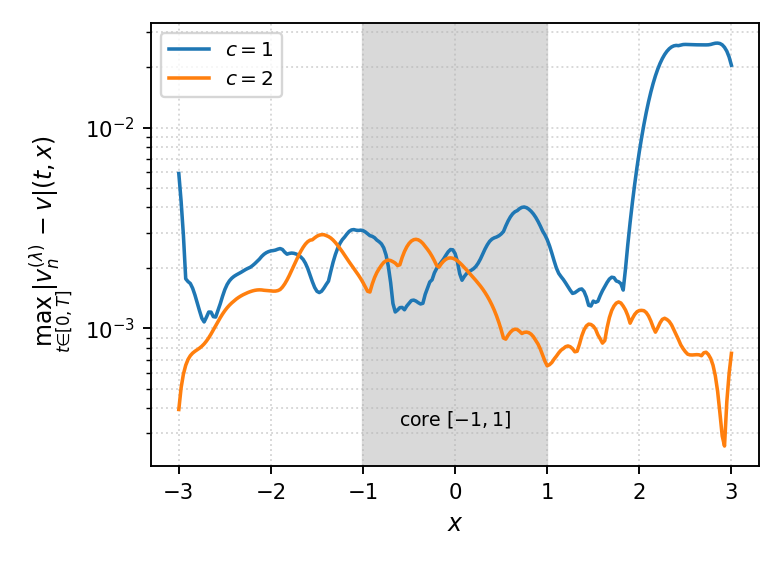}
\end{minipage}
\caption{Kernel realization, $d=1$, $(N^t,N^x)=(32,25)$, $R=3$, $\lambda=6$.
(left) Errors on the core and the a posteriori bound $\gamma_n$ versus the bandwidth
$c$.
(right) Spatial error profile $x\mapsto\max_{t\in[0,T]}|v^{(\lambda)}_n-v|(t,x)$ for
$c\in\{1,2\}$; the shaded band is the core $[-1,1]$ on which the errors of
Table~\ref{tab:ker.refine_d1} are measured.}
\label{fig:ker.sens}
\end{figure}

The fit fails at $c=0.5$ but is relatively insensitive to $c\ge1$, where optimization appears to be
the limiting factor. For $c=1$ the largest errors concentrate near the artificial boundary, whereas
$c=2$ largely removes this layer, illustrating the difference between core and full-cylinder errors.

\FloatBarrier
\smallskip
\noindent\textbf{Theorem-scaled refinement ($d=1$).}
To couple domain truncation with space-time refinement, we take
$R\in\{1.5,2,2.5,3,3.5\}$, $N^t=\lceil R^3\rceil$ and $N^x=N^t+1$, with
$c=1$, $\lambda=6$ and $N^{\mathrm{extra}}=0$. Thus
$h_n=[(1/(2N^t))^2+1/(N^x-1)^2]^{1/2}=O(R^{-3})$, so the geometric factor
$R^2h_n$ in Theorem~\ref{thm:ker.qual} is $O(R^{-1})$. Table~\ref{tab:ker.theory_d1}
shows that the core RMS error decreases from $4.1\times10^{-1}$ to $7.1\times10^{-4}$, while all
computed native space norms remain below $\lambda$. The empirical loss and $\gamma_n$ rise slightly
on the finest grid, indicating an optimization floor. This finite experiment illustrates the grid
and truncation scaling of the theorem, but does not verify the strict condition
$\lambda>\|v\|_{\mathcal H}$ or the asymptotic optimization condition \textnormal{(K2)}.

\begin{table}[!ht]
  \centering
  \small
  \caption{Theorem-scaled refinement of the kernel realization ($d=1$).}
  \label{tab:ker.theory_d1}
  \begin{tabular}{cccccccccc}
    \toprule
    $R_n$ & $N^t_n$ & $N^x_n$ & $N$ & $R_n^2h_n$ & $\widehat{\mathcal L}_n$ & $\gamma_n$ & $\|v_n\|_{\mathcal H}$ & $E_{\mathrm{RMS}}^{\mathrm{core}}$ & CPU [s] \\
    \midrule
    1.5 & 4 & 5 & 25 & 0.6289 & 7.921e-02 & 5.138e-01 & 1.35 & 4.074e-01 & 0.0 \\
    2.0 & 8 & 9 & 81 & 0.5590 & 1.878e-03 & 1.069e-01 & 2.43 & 1.513e-01 & 0.1 \\
    2.5 & 16 & 17 & 289 & 0.4367 & 2.671e-04 & 7.517e-02 & 2.75 & 6.726e-03 & 0.4 \\
    3.0 & 27 & 28 & 784 & 0.3727 & 4.763e-05 & 4.214e-02 & 2.77 & 1.020e-03 & 15.9 \\
    3.5 & 43 & 44 & 1936 & 0.3185 & 5.676e-05 & 6.298e-02 & 2.80 & 7.105e-04 & 69.6 \\
    \bottomrule
  \end{tabular}
\end{table}

\FloatBarrier

\smallskip
\noindent\textbf{Two-dimensional results.}
For $d=2$ we set $R=3$, $N^t=16$, $c=1.5$ and $\lambda=6$, and refine $N^x$ from $7$ to $15$.
Table~\ref{tab:ker.refine_d2} and Figure~\ref{fig:ker.refine.d2} report the results.

\begin{table}[tb]
  \centering
  \small
  \caption{Spatial-grid refinement of the kernel realization ($d=2$).}
  \label{tab:ker.refine_d2}
  \begin{tabular}{cccccccccc}
    \toprule
    $N^t$ & $N^x$ & $N$ & $h_x$ & $\gamma_n$ & $\|v_n\|_{\mathcal H}$ & $E_{\mathrm{RMS}}^{\mathrm{core}}$ & $E_{\infty}^{\mathrm{core}}$ & $E_{\mathrm{RMS}}$ & CPU [s] \\
    \midrule
    16 & 7 & 833 & 1.0000 & 3.129e-01 & 3.11 & 3.190e-01 & 7.054e-01 & 1.665e-01 & 2.9 \\
    16 & 9 & 1377 & 0.7500 & 1.348e-01 & 3.53 & 1.507e-01 & 4.101e-01 & 7.630e-02 & 7.8 \\
    16 & 11 & 2057 & 0.6000 & 4.446e-02 & 3.85 & 3.700e-02 & 1.029e-01 & 1.604e-02 & 33.0 \\
    16 & 13 & 2873 & 0.5000 & 4.057e-02 & 3.92 & 7.212e-03 & 2.165e-02 & 3.657e-03 & 84.6 \\
    16 & 15 & 3825 & 0.4286 & 3.383e-02 & 3.95 & 2.386e-03 & 6.930e-03 & 2.156e-03 & 268.6 \\
    \bottomrule
  \end{tabular}
\end{table}

\begin{figure}[!b]
\centering
\includegraphics[width=0.98\columnwidth]{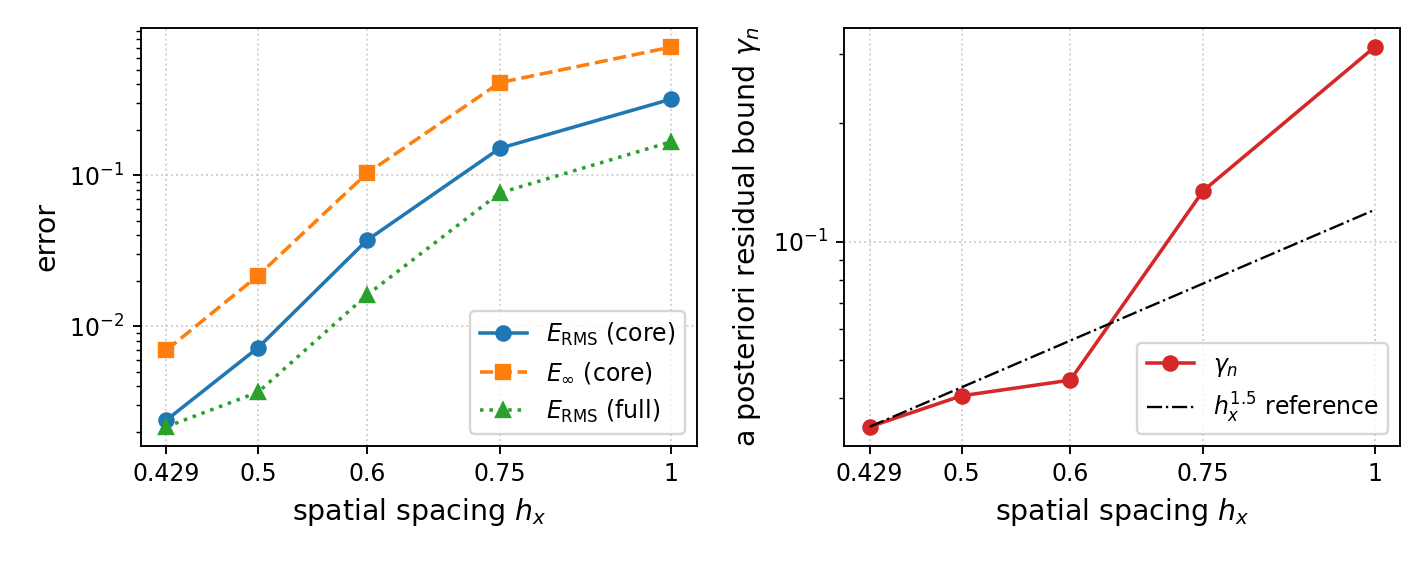}
\caption{Kernel realization, $d=2$, $R=3$, $c=1.5$, $\lambda=6$.
(left) Errors on the core $Q^{\mathrm{core}}=[0,1]\times[-1,1]^2$ and on the full truncation
cylinder, versus the spatial spacing $h_x$.
(right) The a posteriori residual bound $\gamma_n$ versus $h_x$, against an $h_x^{3/2}$
reference slope.}
\label{fig:ker.refine.d2}
\end{figure}

All errors and $\gamma_n$ decrease monotonically; the fitted core rate is pre-asymptotic on this
coarse range, while the two finest values of $\gamma_n$ roughly follow the $h_x^{3/2}$ reference.
The finest dense-kernel solve takes $269$\,s.

\smallskip
\noindent\textbf{Comparison with a monotone finite-difference baseline.}
We also ran a monotone explicit finite-difference discretisation of \eqref{eq:4.testHJB} in $d=1$
(central second differences inside the minimum over the diffusion level, upwind first differences,
and explicit Euler under $\Delta t\,(\sigma_{\max}^2/h_x^2+\bar b/h_x)\le1$). Table~\ref{tab:fd.compare_d1}
uses the same metrics as above and compares exact Dirichlet data (\emph{oracle}) with linear
extrapolation (\emph{extrap}) at $|x|=R$.

On the core the kernel realization is more accurate: $E^{\mathrm{core}}_{\mathrm{RMS}}=6.4\times10^{-4}$
against $2.8\times10^{-3}$ at the finest finite-difference setting. On the full cylinder the
corresponding errors are $1.5\times10^{-3}$ for the kernel method, $2.1\times10^{-3}$ with oracle
boundary data and $5.7\times10^{-3}$ with extrapolation. The two boundary treatments give the same
core errors to the displayed precision but differ near the boundary. The finite-difference method is
far cheaper (about $0.1$\,s versus $71$\,s), highlighting the cost of dense kernel matrices in one
dimension.

\begin{table}[tb]
  \centering
  \small
  \caption{Monotone finite-difference baseline for the $d=1$ test problem (central second differences inside the minimum over the diffusion level, upwind first differences, explicit Euler in $t$ under the CFL condition), evaluated on the same core and full grids and with the same metrics as Table~\ref{tab:ker.refine_d1}. \emph{oracle}: Dirichlet values at $|x|=R$ taken from the analytical solution; \emph{extrap}: linear extrapolation at $|x|=R$, i.e.\ no knowledge of the exact solution. The core errors coincide to the reported precision for the two treatments because the drift has speed at most $1$ and the diffusion coefficient is at most $0.08$.}
  \label{tab:fd.compare_d1}
  \begin{tabular}{llccccc}
    \toprule
    BC & $N^x$ & steps & $E_{\mathrm{RMS}}^{\mathrm{core}}$ & $E_{\infty}^{\mathrm{core}}$ & $E_{\mathrm{RMS}}$ & CPU [s] \\
    \midrule
    oracle & 16 & 5 & 7.127e-02 & 1.863e-01 & 6.116e-02 & 0.000 \\
    oracle & 32 & 12 & 3.688e-02 & 9.127e-02 & 3.019e-02 & 0.001 \\
    oracle & 64 & 33 & 1.933e-02 & 4.692e-02 & 1.543e-02 & 0.003 \\
    oracle & 128 & 105 & 1.025e-02 & 2.436e-02 & 8.035e-03 & 0.008 \\
    oracle & 256 & 370 & 5.352e-03 & 1.256e-02 & 4.153e-03 & 0.037 \\
    oracle & 512 & 1386 & 2.751e-03 & 6.441e-03 & 2.121e-03 & 0.157 \\
    \midrule
    extrap & 16 & 5 & 7.127e-02 & 1.863e-01 & 7.902e-02 & 0.000 \\
    extrap & 32 & 12 & 3.688e-02 & 9.127e-02 & 3.916e-02 & 0.001 \\
    extrap & 64 & 33 & 1.933e-02 & 4.692e-02 & 1.984e-02 & 0.002 \\
    extrap & 128 & 105 & 1.025e-02 & 2.436e-02 & 1.119e-02 & 0.006 \\
    extrap & 256 & 370 & 5.352e-03 & 1.256e-02 & 7.314e-03 & 0.023 \\
    extrap & 512 & 1386 & 2.751e-03 & 6.441e-03 & 5.658e-03 & 0.098 \\
    \bottomrule
  \end{tabular}
\end{table}

\FloatBarrier

\subsection{Neural-network realization: architecture and Sobolev-regularized loss}\label{sec:5.3}

We approximate $v$ by a feedforward $\tanh$ network $v_{\theta}$ as in \eqref{eq:3.1.NN} with $L$
hidden layers of width $W$. Unlike the kernel candidates reported above, which all satisfy the
native space constraint $\theta^{\mathsf T}K\theta\le\lambda^2$ when checked after optimization,
the network candidates are not trained under the hard weight bound or Sobolev-norm constraint used
in Section~\ref{sec:4}; the penalty below is only an experimental regularizer.
Collocation points $\{\tilde y^{\ell}\}_{\ell=1}^M$ are drawn uniformly from $[0,1]\times[-R,R]$,
and terminal points $\{\tilde x^{\ell}\}_{\ell=1}^{M'}$ are drawn uniformly from $[-R,R]$.

\smallskip
\noindent\textbf{Sobolev surrogate.}
Since $\|v_{\theta}\|_{H^s(Q_n)}$ is too costly to evaluate during training, we use the product of
layer norms
\begin{equation}
\label{eq:4.pathnorm}
 \mathcal P(\theta):=\prod_{\ell=1}^{L+1}\|A^{(\ell)}\|_{F},
\end{equation}
where $\|\cdot\|_F$ is the Frobenius norm; products of layer norms of this kind appear in
norm-based capacity bounds \cite{bar-fos-tel:2017,gol-rak-sha:2018}.
Since $\|\varrho'\|_{L^{\infty}}\le1$, $\mathcal P(\theta)$ bounds the Lipschitz constant of $v_{\theta}$.
It does not by itself bound higher derivatives or $\|v_{\theta}\|_{H^s(Q_n)}$: already the one-hidden-layer
network $u_w(x)=w^{-1}\tanh(wx)$, with parameter vector $\theta_w$, satisfies
$\mathcal P(\theta_w)=1$ while its second derivative grows with $w$.
We use $\mathcal P(\theta)$ only as a computable experimental proxy; it does not enforce the constraint in
\eqref{eq:Phatn}, whose left-hand side is measured after training in Section~\ref{sec:5.4}.

\smallskip
\noindent\textbf{Loss.}
We minimize the augmented loss
\begin{equation}
\label{eq:4.loss}
 \mathcal J_n(\theta):=\frac{1}{M}\sum_{\ell=1}^M|r(v_{\theta})(\tilde y^{\ell})|^2
   + \frac{1}{M'}\sum_{\ell=1}^{M'}|v_{\theta}(1,\tilde x^{\ell})-g(\tilde x^{\ell})|^2
   + \beta\,\bigl(\mathcal P(\theta)\bigr)^2,
\end{equation}
where $\beta\ge0$ is an experimental regularization weight. Since the penalty differs from the
Sobolev constraint in \eqref{eq:Phatn}, the convergence theorems of Section~\ref{sec:4} do not
apply directly to these training runs.

\smallskip
\noindent\textbf{Training.}
Derivatives $\partial_t v_{\theta}, \partial_x v_{\theta}, \partial^2_x v_{\theta}$ are computed by
automatic differentiation. The loss \eqref{eq:4.loss} is minimized by Adam~\cite{kin-ba:2015} with learning rate
$10^{-3}$ for $5\times 10^3$ iterations, a fresh batch of collocation and terminal points being
drawn at every step.
Every $10^3$ iterations the RMS error against the exact solution is evaluated on the same held-out set
of $2\times 10^3$ uniform samples, and the iterate with the lowest validation error is retained. This oracle choice, which uses the exact solution, separates the optimization error from the other
error sources.

In all experiments below we set $T=1$, $R=3$, $L=3$, and $M'=\max(M/4,256)$,
and vary the network width $W$, the residual sample size $M$, and the regularization weight $\beta$.

\FloatBarrier

\subsection{Neural-network realization: numerical results}\label{sec:5.4}

The retained iterate is evaluated once on deterministic grids.
On the uniform $51\times 101$ grid $\{z^{j}\}_{j=1}^{M_{\mathrm{eval}}}$ of
$[0,1]\times[-R,R]$, where $M_{\mathrm{eval}}=51\cdot101$, we set
\begin{align*}
 E_{\infty} &:= \max_{1\le j\le M_{\mathrm{eval}}}|v_{\theta}(z^{j})-v(z^{j})|,\qquad
 E_{\mathrm{RMS}} := \Bigl(\tfrac{1}{M_{\mathrm{eval}}}\textstyle\sum_{j=1}^{M_{\mathrm{eval}}}
 |v_{\theta}(z^{j})-v(z^{j})|^2\Bigr)^{1/2},\\
 \mathrm{Res}_{\infty} &:= \max_{1\le j\le M_{\mathrm{eval}}}|r(v_{\theta})(z^{j})|,\qquad
 E_{T,\infty} := \max_{1\le j\le M_{T,\mathrm{eval}}}|v_{\theta}(1,x_T^j)-g(x_T^j)|,
\end{align*}
where $M_{T,\mathrm{eval}}=10^3$ and
$x_T^1,\ldots,x_T^{M_{T,\mathrm{eval}}}$ are equispaced points in $[-R,R]$.
As in Section~\ref{sec:5.2} we also report $E^{\mathrm{core}}_{\mathrm{RMS}}$ and
$E^{\mathrm{core}}_{\infty}$, computed on the same fixed compact core
$Q^{\mathrm{core}}=[0,1]\times[-1,1]$ and the same $21\times 41$ grid, so that the two realizations
are compared on identical sets.
We additionally track $\mathcal P(\theta)$ and the wall-clock training time.

Each setting is trained as described in Section~\ref{sec:5.3}.
Table~\ref{tab:4.main} reports the resulting errors and CPU times for the $(W,M)$ sweep
with $\beta=10^{-6}$ fixed.

\begin{table}[htbp]
\centering
\small
\setlength{\tabcolsep}{4pt}
\caption{Errors and computational cost as a function of width $W$ and sample size $M$, with $\beta=10^{-6}$ fixed.}
\label{tab:4.main}
\begin{tabular}{cccccccccc}
\toprule
$W$ & $M$ & $E_{\mathrm{RMS}}^{\mathrm{core}}$ & $E_{\infty}^{\mathrm{core}}$ & $E_{\mathrm{RMS}}$ & $E_{\infty}$ & $E_{T,\infty}$ & $\mathrm{Res}_{\infty}$ & $\mathcal P(\theta)$ & CPU [s] \\
\midrule
16 & 1024 & 4.393e-03 & 9.493e-03 & 2.956e-03 & 9.490e-03 & 9.402e-03 & 4.378e-02 & 1.99e+01 & 17.7 \\
16 & 4096 & 4.403e-03 & 9.648e-03 & 3.010e-03 & 9.686e-03 & 9.680e-03 & 3.735e-02 & 1.93e+01 & 35.1 \\
32 & 1024 & 2.837e-03 & 8.286e-03 & 3.971e-03 & 2.735e-02 & 9.278e-03 & 4.115e-02 & 1.43e+01 & 24.2 \\
32 & 4096 & 2.132e-03 & 6.792e-03 & 2.659e-03 & 1.586e-02 & 4.983e-03 & 3.615e-02 & 1.25e+01 & 52.0 \\
64 & 1024 & 9.675e-04 & 3.715e-03 & 1.160e-03 & 1.080e-02 & 3.725e-03 & 3.000e-02 & 1.27e+01 & 40.3 \\
64 & 4096 & 8.853e-04 & 2.897e-03 & 1.511e-03 & 1.250e-02 & 4.509e-03 & 3.114e-02 & 1.22e+01 & 104.3 \\
\bottomrule
\end{tabular}
\end{table}

\begin{figure}[tbp]
\centering
\begin{minipage}[b]{0.49\columnwidth}
  \centering
  \includegraphics[width=0.95\linewidth]{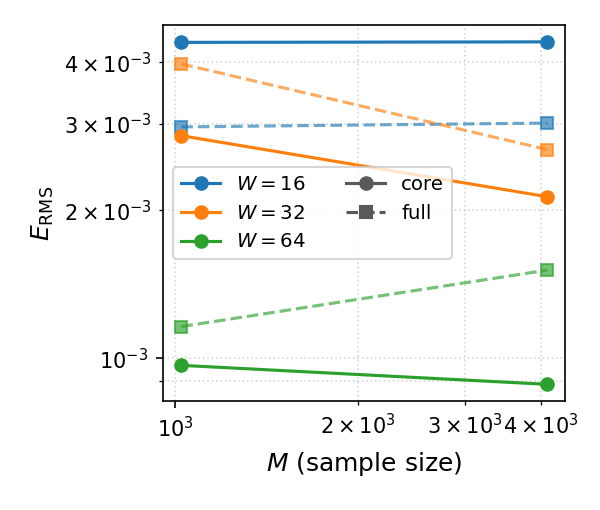}
\end{minipage}\hfill
\begin{minipage}[b]{0.49\columnwidth}
  \centering
  \includegraphics[width=0.95\linewidth]{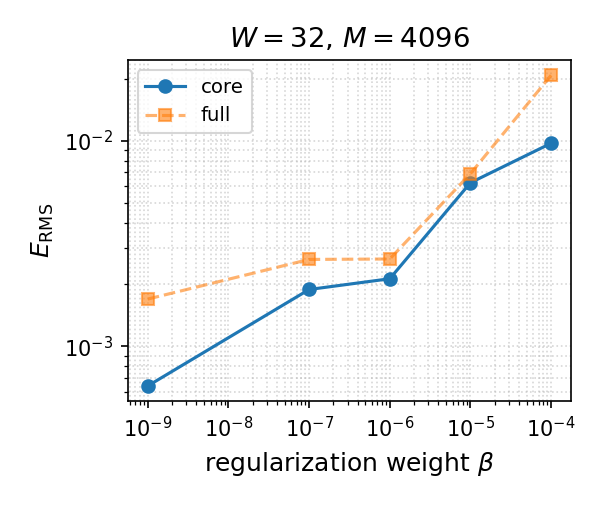}
\end{minipage}
\caption{(left) $E_{\mathrm{RMS}}$ on the core and on the full cylinder versus the sample
size $M$, for $W\in\{16,32,64\}$ and $\beta=10^{-6}$.
(right) The same two errors versus the regularization weight $\beta$ for $(W,M)=(32,4096)$;
the value $\beta=0$ is plotted at $10^{-9}$ on the logarithmic axis.}
\label{fig:4.errors}
\end{figure}

\smallskip
\noindent\textbf{Discussion.}
The error decreases with width, whereas changing $M$ has little systematic effect at this training
budget. Increasing $\beta$ worsens both errors, so for this smooth target the experimental penalty
does not improve accuracy.

Post-training quadrature of the retained checkpoints gives $H^5$-norms between $25.1$ and $37.7$
for the width sweep, so each computed network is feasible for an a posteriori radius at least as
large as its measured norm. These checks do not establish the
training constraint or the approximation hypothesis of Theorem~\ref{thm:3.2}. Moreover, increasing
$\beta$ lowers $\mathcal P(\theta)$ but can increase the measured $H^5$-norm, confirming that the
penalty is not a Sobolev constraint.

On the common core the kernel has slightly smaller RMS error than the best regularized network, while
the network has slightly smaller maximum error; their full-cylinder RMS errors are comparable. This
comparison favors the networks through the oracle iterate selection of Section~\ref{sec:5.3}.

These experiments illustrate, rather than validate, the theory: they use one manufactured equation,
$d\le2$, fixed time grids and short training budgets, and the PINN loss does not enforce the theoretical
Sobolev constraint. Higher-dimensional experiments and broader baselines are left for future work.

\FloatBarrier

\subsection*{Acknowledgements}

This study is supported by JSPS KAKENHI Grant Number JP24K06861.

\subsection*{Code availability}

The code, numerical data, and scripts used to generate the figures and tables
in Section~\ref{sec:5} are publicly available at
\url{https://github.com/yumiharu-nakano/hjb-residual-methods}.

\subsection*{AI disclosure statement}

During the preparation of this manuscript, the author used OpenAI Codex
(accessed in 2026) to refactor and document Python code, generate routine
plotting code, and edit author-written prose. The author critically reviewed
and tested all AI-assisted material and independently verified the mathematical
arguments, experimental design, numerical results, and conclusions. The author
takes full responsibility for the content of this work.

\begingroup
\raggedright
\bibliographystyle{plain}
\bibliography{../mybib,approxpde601_recent}

\newcommand{\noop}[1]{}
\begin{thebibliography}{10}

\bibitem{aki-fod:2018}
M.~Akian and E.~Fodjo.
\newblock From a monotone probabilistic scheme to a probabilistic max-plus
  algorithm for solving {Hamilton--Jacobi--Bellman} equations.
\newblock In D.~Kalise, K.~Kunisch, and Z.~Rao, editors, {\em
  Hamilton--Jacobi--Bellman Equations: Numerical Methods and Applications in
  Optimal Control}, volume~21 of {\em Radon Series on Computational and Applied
  Mathematics}, pages 1--24. De Gruyter, 2018.

\bibitem{bar-sou:1991}
G.~Barles and P.~E. Souganidis.
\newblock Convergence of approximation schemes for fully nonlinear second order
  equations.
\newblock {\em Asymptot.~Anal.}, 4:271--283, 1991.

\bibitem{bar-fos-tel:2017}
P.~L. Bartlett, D.~J. Foster, and M.~J. Telgarsky.
\newblock Spectrally-normalized margin bounds for neural networks.
\newblock {\em Adv.~Neural Inf.~Process.~Syst.}, 30, 2017.

\bibitem{bar-men:2002}
P.~L. Bartlett and S.~Mendelson.
\newblock {R}ademacher and {G}aussian complexities: risk bounds and structural
  results.
\newblock {\em J. Mach. Learn. Res.}, 3:463--482, 2002.

\bibitem{bok-pic-rei:2021}
O.~Bokanowski, A.~Picarelli, and C.~Reisinger.
\newblock Stability and convergence of second order backward differentiation
  schemes for parabolic {Hamilton--Jacobi--Bellman} equations.
\newblock {\em Numer.~Math.}, 148:187--222, 2021.

\bibitem{bon-dev-pet-sie:2026}
A.~Bonito, R.~DeVore, G.~Petrova, and J.~W. Siegel.
\newblock Convergence and error control of consistent {PINN}s for elliptic
  {PDE}s.
\newblock {\em {IMA} J.~Numer.~Anal.}, 46(1):90--148, 2026.

\bibitem{bon-zid:2003}
F.~Bonnans and H.~Zidani.
\newblock Consistency of generalized finite difference schemes for the
  stochastic {HJB} equation.
\newblock {\em {SIAM} J.~Numer.~Anal.}, 41:1008--1021, 2003.

\bibitem{cam-fal:1995}
F.~Camilli and M.~Falcone.
\newblock An approximation scheme for the optimal control of diffusion
  processes.
\newblock {\em {RAIRO} Mod\'el.~Math.~Anal.~Num\'er.}, 29(1):97--122, 1995.

\bibitem{drm:2021}
T.~De~Ryck, S.~Lanthaler, and S.~Mishra.
\newblock On the approximation of functions by tanh neural networks.
\newblock {\em Neural Netw.}, 143:732--750, 2021.

\bibitem{dr-mis:2022}
T.~De~Ryck and S.~Mishra.
\newblock Error analysis for physics-informed neural networks ({PINN}s)
  approximating {K}olmogorov {PDE}s.
\newblock {\em Adv. Comput. Math.}, 48(6):79, 2022.

\bibitem{dr-mis:2024}
T.~De~Ryck and S.~Mishra.
\newblock Numerical analysis of physics-informed neural networks and related
  models in physics-informed machine learning.
\newblock {\em Acta Numer.}, 33:633--713, 2024.

\bibitem{deb-jak:2012}
K.~Debrabant and E.~R. Jakobsen.
\newblock Semi-{L}agrangian schemes for linear and fully non-linear diffusion
  equations.
\newblock {\em Math.~Comp.}, 82:1433--1462, 2013.

\bibitem{e-etal:2017}
W.~E, J.~Han, and A.~Jentzen.
\newblock Deep learning-based numerical methods for high-dimensional parabolic
  partial differential equations and backward stochastic differential
  equations.
\newblock {\em Commun.~Math.~Stat.}, 5:349--380, 2017.

\bibitem{ehr-azm-haa:2026}
T.~Ehring, B.~Azmi, and B.~Haasdonk.
\newblock Recovery of the optimal control value function in reproducing kernel
  {Hilbert} spaces from verification conditions.
\newblock {\em Math. Control Signals Systems}, 2026.
\newblock Published online 26 August 2026.

\bibitem{fah-tou-war:2011}
A.~Fahim, N.~Touzi, and X.~Warin.
\newblock A probabilistic numerical method for fully nonlinear parabolic
  {PDE}s.
\newblock {\em Ann.~Appl.~Probab.}, 21:1322--1364, 2011.

\bibitem{fen-glo-nei:2013}
X.~Feng, R.~Glowinski, and M.~Neilan.
\newblock Recent developments in numerical methods for fully nonlinear second
  order partial differential equations.
\newblock {\em {SIAM} Rev.}, 55(2):205--267, 2013.

\bibitem{fen-lew:2021}
X.~Feng and T.~Lewis.
\newblock A narrow-stencil finite difference method for approximating viscosity
  solutions of {Hamilton--Jacobi--Bellman} equations.
\newblock {\em {SIAM} J. Numer. Anal.}, 59(2):886--924, 2021.

\bibitem{fle-son:2006}
W.~H. Fleming and H.~M. Soner.
\newblock {\em Controlled {M}arkov processes and viscosity solutions}.
\newblock Springer-Verlag, New York, 2nd edition, 2006.

\bibitem{gal-tra:2024}
D.~Gallistl and N.~T. Tran.
\newblock Minimal residual discretization of a class of fully nonlinear
  elliptic {PDE}.
\newblock {\em {IMA} J. Numer. Anal.}, 2025.
\newblock Published online; {DOI} 10.1093/imanum/draf075.

\bibitem{gaz-gka-mak:2025}
D.~Gazoulis, I.~Gkanis, and C.~G. Makridakis.
\newblock On the stability and convergence of physics informed neural networks.
\newblock {\em {IMA} J.~Numer.~Anal.}, 2025.
\newblock Published online 28 October 2025.

\bibitem{gol-rak-sha:2018}
N.~Golowich, A.~Rakhlin, and O.~Shamir.
\newblock Size-independent sample complexity of neural networks.
\newblock In {\em Proc.~31st Conf.~Learn.~Theory ({COLT})}, volume~75 of {\em
  Proc.~Mach.~Learn.~Res.}, pages 297--299, 2018.

\bibitem{guo-etal:2015}
W.~Guo, J.~Zhang, and J.~Zhuo.
\newblock A monotone scheme for high-dimensional fully nonlinear {PDE}s.
\newblock {\em Ann.~Appl.~Probab.}, 25:1540--1580, 2015.

\bibitem{han-jen-e:2018}
J.~Han, A.~Jentzen, and W.~E.
\newblock Solving high-dimensional partial differential equations using deep
  learning.
\newblock {\em Proc. Natl. Acad. Sci. {USA}}, 115(34):8505--8510, 2018.

\bibitem{hoe:1963}
W.~Hoeffding.
\newblock Probability inequalities for sums of bounded random variables.
\newblock {\em J.~Amer.\ Statist.\ Assoc.}, 58:13--30, 1963.

\bibitem{hof-sun-coh:2024}
W.~Hofgard, J.~Sun, and A.~Cohen.
\newblock Convergence of the deep galerkin method for finite state mean field
  control problems, 2024.
\newblock arXiv:2405.13346.

\bibitem{hor:1991}
K.~Hornik.
\newblock Approximation capabilities of multilayer feedforward networks.
\newblock {\em Neural Networks}, 4:251--257, 1991.

\bibitem{hor-etal:1990}
K.~Hornik, M.~Stinchcombe, and H.~White.
\newblock Universal approximation of an unknown mapping and its derivatives
  using multilayer feedforward networks.
\newblock {\em Neural Networks}, 3:551--560, 1990.

\bibitem{bac-hur-pha-lan:2022}
C.~Hur\'e, H.~Pham, A.~Bachouch, and N.~Langren\'e.
\newblock Deep neural networks algorithms for stochastic control problems on
  finite horizon: convergence analysis.
\newblock {\em {SIAM} J. Numer. Anal.}, 59(1):525--557, 2021.

\bibitem{hur-pha-war:2020}
C.~Hur\'e, H.~Pham, and X.~Warin.
\newblock Deep backward schemes for high-dimensional nonlinear {PDE}s.
\newblock {\em Math. Comp.}, 89(324):1547--1579, 2020.

\bibitem{hut-jen-kru-ngu:2020}
M.~Hutzenthaler, A.~Jentzen, T.~Kruse, and T.~A. Nguyen.
\newblock A proof that rectified deep neural networks overcome the curse of
  dimensionality in the numerical approximation of semilinear heat equations.
\newblock {\em {SN} Partial Differ. Equ. Appl.}, 1(2):10, 2020.

\bibitem{kan:1990b}
E.~J. Kansa.
\newblock Multiquadrics---a scattered data approximation scheme with
  applications to computational fluid-dynamics---{II}.
\newblock {\em Computers Math.~ Applic.}, 19:147--161, 1990.

\bibitem{kim-etal:2026pi}
Y.~Kim, M.~Kim, Y.~Kim, and N.~Cho.
\newblock Physics-informed policy iteration for high-dimensional
  {Hamilton--Jacobi--Bellman} equations: Interior error bounds without boundary
  data, 2026.
\newblock arXiv:2508.01718, version 2.

\bibitem{kin-ba:2015}
D.~P. Kingma and J.~Ba.
\newblock {A}dam: a method for stochastic optimization.
\newblock In {\em Proc.~Int.~Conf.~Learn.~Repr.~({ICLR})}, 2015.

\bibitem{kry:1987}
N.~V. Krylov.
\newblock {\em Nonlinear elliptic and parabolic equations of the second order}.
\newblock D. Reidel Publishing Co., Dordrecht, 1987.

\bibitem{kus-dup:2001}
H.~J. Kushner and P.~Dupuis.
\newblock {\em Numerical methods for stochastic control problems in continuous
  time}.
\newblock Springer-Verlag, New York, 2nd edition, 2001.

\bibitem{liu-noc:1989}
D.~C. Liu and J.~Nocedal.
\newblock On the limited memory {BFGS} method for large scale optimization.
\newblock {\em Math.~Program.}, 45:503--528, 1989.

\bibitem{liu:2026}
J.~Liu.
\newblock Verifiable error bounds for physics-informed neural network solutions
  of {Lyapunov} and {Hamilton--Jacobi--Bellman} equations.
\newblock {\em {IEEE} Control Syst. Lett.}, 10:403--408, 2026.

\bibitem{men-etal:2024}
Y.~Meng, R.~Zhou, A.~Mukherjee, M.~Fitzsimmons, C.~Song, and J.~Liu.
\newblock Physics-informed neural network policy iteration: Algorithms,
  convergence, and verification.
\newblock In {\em Proceedings of the 41st International Conference on Machine
  Learning}, volume 235 of {\em Proc. Mach. Learn. Res.}, pages 35378--35403,
  2024.

\bibitem{mis-mol:2022}
S.~Mishra and R.~Molinaro.
\newblock Estimates on the generalization error of physics-informed neural
  networks for approximating a class of inverse problems for {PDE}s.
\newblock {\em IMA J.~Numer.~Anal.}, 42(2):981--1022, 2022.

\bibitem{nak:2014b}
Y.~Nakano.
\newblock An approximation scheme for stochastic controls in continuous time.
\newblock {\em Jpn.~J.~Ind.~Appl.~Math.}, 31:681--696, 2014.

\bibitem{nak:2017}
Y.~Nakano.
\newblock Convergence of meshfree collocation methods for fully nonlinear
  parabolic equations.
\newblock {\em Numer.~Math.}, 136:703--723, 2017.
\newblock (See also {\tt arxiv:1408.5195[math.NA]}.)\noop{.}

\bibitem{nak:2019a}
Y.~Nakano.
\newblock Kernel-based collocation methods for {Z}akai equations.
\newblock {\em Stoch. Partial Differ. Equ. Anal. Comput.}, 7:476--494, 2019.
\newblock (See also {\tt arxiv:1710.09090[math.NA]}.)\noop{.}

\bibitem{nak:2025x}
Y.~Nakano.
\newblock Convergence of differentiable non-monotone schemes for fully
  nonlinear parabolic equations.
\newblock arXiv:1803.09446, 2026.

\bibitem{nei-sal-zha:2017}
M.~Neilan, A.~J. Salgado, and W.~Zhang.
\newblock Numerical analysis of strongly nonlinear {PDE}s.
\newblock {\em Acta Numer.}, 26:137--303, 2017.

\bibitem{pag-pha-pri:2004}
G.~Pag{\`e}s, H.~Pham, and J.~Printems.
\newblock An optimal {M}arkovian quantization algorithm for multidimensional
  stochastic control problems.
\newblock {\em Stoch. Dyn.}, 4:501--545, 2004.

\bibitem{pha:2009}
H.~Pham.
\newblock {\em Continuous-time stochastic control and optimization with
  financial applications}.
\newblock Springer, Berlin, 2009.

\bibitem{rai-per-kar:2019}
M.~Raissi, P.~Perdikaris, and G.~E. Karniadakis.
\newblock Physics-informed neural networks: a deep learning framework for
  solving forward and inverse problems involving nonlinear partial differential
  equations.
\newblock {\em J.~Comput.~Phys.}, 378:686--707, 2019.

\bibitem{rei-rot:2017}
C.~Reisinger and J.~Rotaetxe~Arto.
\newblock Boundary treatment and multigrid preconditioning for semi-lagrangian
  schemes applied to {Hamilton--Jacobi--Bellman} equations.
\newblock {\em J.~Sci.~Comput.}, 72:198--230, 2017.

\bibitem{sir-spi:2018}
J.~Sirignano and K.~Spiliopoulos.
\newblock {DGM}: A deep learning algorithm for solving partial differential
  equations.
\newblock {\em J.~Comput.~Phys.}, 375:1339--1364, 2018.

\bibitem{sme-sul:2013}
I.~Smears and E.~S\"uli.
\newblock Discontinuous {G}alerkin finite element approximation of
  nondivergence form elliptic equations with {C}ord\`es coefficients.
\newblock {\em {SIAM} J. Numer. Anal.}, 51(4):2088--2106, 2013.

\bibitem{sme-sul:2014}
I.~Smears and E.~S\"uli.
\newblock Discontinuous {G}alerkin finite element approximation of
  {H}amilton--{J}acobi--{B}ellman equations with {C}ordes coefficients.
\newblock {\em {SIAM} J. Numer. Anal.}, 52(2):993--1016, 2014.

\bibitem{sme-sul:2016}
I.~Smears and E.~S\"uli.
\newblock Discontinuous {G}alerkin finite element methods for time-dependent
  {H}amilton--{J}acobi--{B}ellman equations with {C}ordes coefficients.
\newblock {\em Numer. Math.}, 133(1):141--176, 2016.

\bibitem{wen:2005}
H.~Wendland.
\newblock {\em Scattered data approximation}.
\newblock Cambridge University Press, Cambridge, 2005.

\bibitem{zei-mas-mar:2025}
M.~Zeinhofer, R.~Masri, and K.-A. Mardal.
\newblock A unified framework for the error analysis of physics-informed neural
  networks.
\newblock {\em {IMA} J.~Numer.~Anal.}, 45(5):2988--3025, 2025.

\end{thebibliography}
\endgroup

\end{document}